 \documentclass[11pt,twoside,reqno,centertags,draft]{amsart}

 \PII{} 
\copyrightinfo{}{}

\thanks{AMS Subject Classifications:  35B35, 35Q55, 35B40, 35A05}

 \usepackage{amsmath,amsthm,amsfonts,amssymb}
\usepackage[mathscr]{euscript}
 \usepackage{hyperref} 
 \usepackage{mathrsfs} 
\usepackage{graphicx}
\usepackage{color}
  \usepackage{epsfig}

\usepackage[ansinew]{inputenc}
\usepackage{comment} 
\usepackage{latexsym}     
 \usepackage{tikz}
\usepackage{graphicx}
 
\usepackage{tikz-3dplot}

\newtheorem{theorem}{Theorem}

\newtheorem{lemma}[theorem]{Lemma}
\newtheorem{proposition}[theorem]{Proposition}
\newtheorem{definition}[theorem]{Definition}
\newtheorem{remark}[theorem]{Remark}

\newcommand{\beqaa}{\begin{eqnarray}}
\newcommand{\eeqaa}{\end{eqnarray}}
\newcommand{\beqae}{\begin{eqnarray*}}
\newcommand{\eeqae}{\end{eqnarray*}}

\newcommand{\con}{\rightarrow}

\newcommand{\pbrack}[1]{\left( {#1} \right)}

\newcommand{\key}[1]{\left\{ {#1} \right\}}
\newcommand{\dual}[2]{\langle{#1},{#2}\rangle}

\newcommand{\R}{{\mathbb R}}
\newcommand{\N}{{\mathbb N}}

\usepackage{hyperref}

\begin{document}
 
\title{ The regularity of the coupled system between an 
electrical network with fractional 
dissipation and a plate equation with 
fractional inertial rotational }
\thanks{Accepted for publication: June 2023.} 
\date{}
\maketitle     
 
\vspace{ -1\baselineskip}

{\small
\begin{center}
{\sc Santos R.W.S.  Bejarano}\\
Department of Mathematics \\  
Federal University of Technology of  Paran\'a, Brazil \\[10pt]
{\sc Filomena B.R.  Mendes}\\
Department of Electrical Engineering \\ 
Federal University of Technology of Paran\'a, Brazil  \\[10pt]
{\sc Fredy M.S.  Su\'arez}\footnote{Corresponding author: fredy@utfpr.edu.br},  
{\sc Gilson Tumelero,}  and \\
{\sc Marieli M.  Tumelero}\\
Department of Mathematics\\  
Federal University of Technology of  Paran\'a, Brazil  \\[10pt]
 (Submitted by: Patrizia Pucci)  
\end{center}
}

\numberwithin{equation}{section}
\allowdisplaybreaks

 \smallskip

 \begin{quote}
\footnotesize
{\bf Abstract.}  
In this work we study a strongly coupled 
system between the equation of plates with 
fractional rotational inertial force 
$\kappa(-\Delta)^\beta u_{tt}$ where the parameter 
$0 <\beta\leq 1$ and the equation of an 
electrical network containing a fractional 
dissipation term 
$\delta(-\Delta)^\theta v_t$ where the parameter 
$0\leq \theta\leq 1$, the strong coupling 
terms are given by the Laplacian of the 
displacement speed  
$\gamma \Delta u_t$ 
and the Laplacian electric potential field 
$\gamma\Delta v_t$.  
When $\beta = 1$, 
we have the Kirchoff-Love plate and when 
$\beta = 0$,  we have the 
Euler-Bernoulli plate recently studied in 
Su\'arez-Mendes (2022-Preprinter)\cite{Suarez}. 
The contributions of this research
are: We prove the semigroup $S(t)$ 
associated with the system is not analytic in 
$(\theta,\beta)\in [0,1]\times(0,1]-\{( 1,1/2)\}$.  
We also determine two Gevrey classes: 
 $s_1 >\frac{1}{2\max\{ \frac{1-\beta}{3-\beta}, \frac{\theta}{2+\theta-\beta}\}}$  for  $2\leq \theta+2\beta$ and 
$s_2>  \frac{2(2+\theta-\beta)}{\theta}$  
when the parameters $\theta$ and 
$\beta$  lies in the interval $(0, 1)$  
and we finish by proving that at the point 
$(\theta,\beta)=(1,1/2)$ the semigroup 
$S(t)$ is analytic and with a note about 
the asymptotic behavior of $S(t)$. 
We apply semigroup theory, the frequency 
domain method together with multipliers 
and the proper decomposition of the 
system components and Lions'  
interpolation inequality.
\end{quote}
 
\setcounter{equation}{0}
\section{Introduction}

This research was motivated by the 
mathematical model used to control the 
vibrations of plates coupled to an electrical 
network composed of piezoelectric actuators 
and/or transducers,  see  equations
$(2b)$ and
$(2c)$ (page 1247) of \cite{CMIsolaDDVescovo}, 
they   represent the mathematical model of 
an electrical network called: Second-order 
electrical transmission line with zero-order 
or second-order dissipation, respectively.  
They  are given by 
\begin{equation*}
v_{tt}-\beta_2\Delta v+\delta_0v_t =0 
\qquad  {\rm and}\qquad 
v_{tt}-\beta_2\Delta v-\delta_2 \Delta v_t =0.
\end{equation*}
 Where
$v(x, t)$ denotes the time integral of 
the electric potential difference between 
the nodes and the ground.

For more details of the mathematical 
model of the plate coupled with a membrane consult 
\cite{VidoliIsola} (Equation (23) page 442)  or  
\cite{PIsolaS2005} (Subsection 3.3 equation 
(18)) satisfying the  boundary conditions 
$u=\Delta u=v=0$ for
$(x,t)\in (\partial\Omega, \mathbb{R}^+)$   
is given by
\begin{eqnarray}
\label{ISplacas0}
    u_{tt}+\beta_2\Delta^2 u-\gamma\Delta
     v_t=0,\quad
    &x\in\Omega,&
t>0, \\
\label{ISplacas1}
    v_{tt}
    -\beta_2\Delta v+\gamma\Delta u_t+\delta 
    v_t+\delta\gamma\Delta u =0,\quad
    &x\in\Omega,&
     t>0,
\end{eqnarray}
where
$u(x,t)$ is   the transversal 
displacements of the plates  and
$v(x,t)$ is time-integral of  the 
electric potential difference  between 
the nodes and the ground, at time
$t>0$  and
$x\in \Omega\subset \R^n$    is a 
bounded open set with smooth boundary
$\partial\Omega$,
the coefficient
$ \beta_2>0$ and 
$\gamma$ coupling coefficient is 
non-zero.  

Another equivalent and well  
known model in the environment of 
mathematical  is the thermoelastic 
Euler-Bernoulli plates of type III.  
Specifically to this last model,  
there are several researches on asymptotic 
properties and analyticity,  two of them are 
\cite{LiuQuintanilla2010, ZOroPata2013}.  
The first one using a characterization 
of semigroups and proving the analyticity 
of the semigroup associated to the system.  
The second one,   defining
$A \colon D(A)\subset H\to H$ being a 
strictly positive self-adjoint operator,
$H$ being the  Hilbert space and considering
$D(A)$ dense in
$H$. They studied  the exponential stability 
of both  models,  the first one is given by
\begin{eqnarray}
\label{ISplacas0A}
u_{tt}+A^2u-A^\sigma v_t=0,\\
\label{ISplacas1A}
v_{tt}+Av+A^\sigma u_t+A^\sigma v_t=0, 
\end{eqnarray}
and the second one:
\begin{eqnarray*}
u_{tt}+A^2u- A^\sigma v_t=0,\\
v_{tt}+Av+A^\sigma u_t+\int_0^\infty 
\mu(s)A[v(t)-v(t-s)]ds=0
\end{eqnarray*}
where
$\sigma\leq\frac{3}{2}$.  
Note that if in the system 
\eqref{ISplacas0A}-\eqref{ISplacas1A} we take
$A=-\Delta$ such that
$D(A)=H_0^1(\Omega)\cap H^2(\Omega)$ and 
$\sigma=1$,  we will have a system 
equivalent to the system 
\eqref{ISplacas0}-\eqref{ISplacas1}.  When
$\sigma\geq 1/2$,  it is  
proved the exponential decay of
$S(t)$ using the energy method.  
It is also proved that
$S(t)$ fails to be exponentially stable if
$\sigma <1/2$.    
From   a physical point of view,  
this is not surprising.  Indeed, 
since the dissipation mechanism is only 
thermal, when the coupling is not strong 
enough, then  the system is not able to 
convert thermal dissipation into mechanical 
dissipation, which is needed to stabilize 
the plate.  Concerning the semigroup 
$S(t)$,  the lack of exponential stability 
is tested for any value of
$\sigma$.

As we can see in the previous systems, 
the first equation is the Euler-Bernoulli 
plate equation.    Recently in Su\'arez 
and Mendes \cite{Suarez}.  The authors 
investigated the model given by:
\begin{eqnarray}
\label{Suarez001}
u_{tt}+\alpha\Delta^2 u-\gamma \Delta v_t 
& = & 0, \quad x\in\Omega,\quad t>0\\
\label{Suarez002}
v_{tt}-\beta \Delta v+\gamma \Delta u_t+
\delta (-\Delta)^\theta v_t & =  & 0, 
\quad x\in \Omega,\quad t>0,
\end{eqnarray}
satisfying the boundary conditions
\begin{equation}
\label{Suarez003}
u=\Delta u=0, \qquad v=0, \qquad  x\in 
\partial\Omega,\qquad t>0.
\end{equation}
In that paper is used frequency domain 
techniques, characterization of semigroup 
theory and Lions interpolation inequality,  
to study asymptotic behavior, analyticity 
and Gevrey class and demonstrate the 
exponential decay of the semigroup
$S(t)=e^{\mathcal{B}t}$,  for
$0 \leq  \theta\leq 1$, it  also proves  
the lack of analyticity of the
semigroup
$S(t)$ for
$\theta\in [0, 1)$ and the analyticity of
$S(t)$ for
$\theta =1$.   Furthermore,   for
$\theta\in (0,1)$ is shown that
$S(t)$ is of Gevrey class
$s>\frac{2+\theta}{\theta}$ for
$\theta\in(0,1)$.

In order to complement the research of 
Su\'arez and Mendes \cite{Suarez} and also 
extend the research with regard to the 
physical model given in 
\eqref{ISplacas0}-\eqref{ISplacas1},  
we introduce the fractional rotational 
inertial force term given by
$\kappa(-\Delta)^\beta u_{tt}$ where 
the parameter
$\beta\in (0,1]$ and we will consider 
the most general dissipative term
$\delta(-\Delta)^\theta v_t$ in the 
electrical mesh equation where the parameter
$\theta\in [0,1]$.  

It is worth mentioning here the recent 
work by G. F. Tyszka et al.  \cite{TMH2022}, 
in this work the authors study the asymptotic 
behavior of the solutions of a dissipative 
coupled system between the plates: 
Kirchhoff-plate and Euler-Bernoulli. 
The dissipation mechanism is given by 
a memory term that acts on both equations 
of the system or just one of the equations. 
The authors show that the solutions of 
the system with dissipation in both 
of the equations (direct dissipations) 
decay exponentially to zero and when 
the dissipation acts only in one of 
the equations of the system 
(indirect dissipation) the solutions of 
the system decay polynomially to zero,  
and in  this case the decay rates 
determined are the best.

 Moreover, to take advantage of the 
 richness of the properties of 
 self-adjoint and positive definite 
 operators,  let us formally define the
$A$ operator as follows:
$A:D(A)\subset L^2(\Omega)\con L^2(\Omega)$, 
where
\begin{equation}\label{Omenoslaplaciano}
A:=-\Delta,\quad D(A)=H^2(\Omega)\cap H^1_0(\Omega).
\end{equation}
It is known that this  operator given in 
\eqref{Omenoslaplaciano} is self-adjoint, 
positive, has compact inverse and it has 
compact resolvent.  Using this operator, 
the system in  investigation is given by
\begin{eqnarray}
\label{ISplacas-20}
    u_{tt}+\kappa A^\beta u_{tt}+ \alpha A^2 u
    +\gamma Av_t=0, \quad x\in\Omega,\quad t>0,\\
\label{ISplacas-25}
    v_{tt}+\alpha Av-\gamma Au_t+\delta 
    A^{\theta} v_t =0,\quad x\in\Omega,\quad t>0,
\end{eqnarray}
the initial data is
\begin{eqnarray}
\label{ISplacas-30} u(0)=u_0,\ u_t(0)=u_1,\
v(0)=v_0,\ v_t(0)=v_1,
\end{eqnarray}
and satisfying the boundary conditions
\begin{equation}\label{ISplacas-40}
u(x,t)=Au(x,t)=0,\quad v(x,t)=0,
\quad{\rm onto}\quad  x\in\partial\Omega,\ t>0.
\end{equation}
 Our system in this investigation is a 
 coupled system of a wave equation with 
 a plate equation with indirect damping 
within the domain (terminology initially 
used by Russell in his paper \cite{Russell-1993}). 
In recent decades, many researchers have 
focused their efforts on studying the 
asymptotic behavior and/or regularity 
of equations with fractional dissipation 
or coupled systems with indirect dam-ping.  
Below,  we briefly mention some research 
in this direction:  Fatori et al.  
\cite{FMJ2012},  studied the differentiability, 
analyticity,  and the optimal decay rate of 
an abstract  wave-like equation given by:
\begin{equation*}
u_{tt}+Au+Bu_t=0,
\end{equation*}
where
$A$, and
$B$ are a self-adjoint positive 
definite operators with the domain
$D(A^\alpha)=D(B)$ dense in a Hilbert space
$H$.  Considering the following hypotheses:
\begin{enumerate}
\item There exists positive constants
$C_1$ and
$C_2$,  and for any
$u\in D(A^\alpha)$ such that 
\begin{equation*}
C_1A^\alpha\leq B\leq C_2A^\alpha
\Longleftrightarrow C_1\dual{A^\alpha u}{u}
\leq\dual{BU}{u}\leq C_2\dual{A^\alpha u}{u}, 
\end{equation*}
\item The bilinear form
$b(u,w)=\dual{B^\frac{1}{2}u}{B^\frac{1}{2}w}$ 
is continuous on
$D(A^\frac{\alpha}{2})\times D(A^\frac{\alpha}{2})$.  
By the Riesz representation theorem, 
this hypothesis implies that there 
is an operator
$S\in \mathcal{L}(D(A^\frac{\alpha}{2}))$ 
such that 
$\dual{Bu}{w}=\dual{A^\frac{\alpha}{2}Su}{A^\frac{\alpha}{2}w}$ 
for any
$u,w\in D(A^\frac{\alpha}{2}).$
\end{enumerate}
They show that: When the parameter
$\alpha \in [1/2,1]$,  the semigroup
$S(t)$ associated with the model is analytic,  
in the case of the parameter
$\alpha \in (0, 1/2 )$ showed that
$S(t)$ is differentiable and also showed 
that for
$\gamma=-\alpha> 0$, the semigroup
$S(t)$  decays polynomially to zero with the 
rate
$t^{-\frac{1}{\gamma}}$ and they still prove 
that this rate is optimal. 

Tebou \cite{Tebou2012} considered a weakly 
coupled system of plate-wave equations with 
indirect  frictional damping  mechanisms.  
The work  showed that the  system is not 
exponentially stable when the damping acts 
either in the plate equation or in the wave 
equation and a polynomial decay of the 
semigroup was showed using a frequency domain 
approach combined with multiplier techniques  
resulting  in the 
characterization of semigroups. 

 Han and Liu in \cite{HanLiu} have 
 recently studied the regularity and 
 asymptotic behavior of two-plate system 
 solutions where only one of them is 
 dissipative and the indirect system 
 dissipation occurs through the higher 
 order coupling term
$\gamma\Delta w_t$ and
$-\gamma\Delta u_t$. The damping mechanism 
considered in this work was the strutural or  
Kelvin-Voigt damping.  More precisely, 
the system studied in \cite{HanLiu} is:
\begin{eqnarray*}
    u_{tt}+\Delta^2 u+\gamma\Delta w_t=0,
    \quad
    &x\in\Omega,&
    t>0, \\
    w_{tt}+\Delta^2 w-\gamma \Delta u_t-d_{st}\Delta  
    w_t+d_{kv}\Delta^2w_t =0,\quad
    &x\in\Omega,&
    t>0,
\end{eqnarray*}
satisfying the  boundary conditions 
$
u=\frac{\partial u}{\partial\nu}\big|_{\partial\Omega}=0,
\quad w=\frac{\partial w}{\partial\nu}
\big|_{\partial\Omega}=0,\quad \ t>0$, where
$u(x,t)$,
$w(x,t)$ denote the transversal displacements 
of the plates at time
$t$ in 
$\Omega$  bounded open set of
$\R^n$ with smooth boundary
$\partial\Omega$, 
$\gamma\not=0$ is the coupling coefficient.   
Only one of the damping coefficients
$d_{st}\geq 0$ and
$d_{kv}\geq 0$ is positive.  Han and Liu 
showed that the semigroup associated with 
the indirect structural damping ($d_{st}>0$ 
and
$d_{kv}=0$) system is analytical
and exponentially stable from the 
frequency domain method. However, 
through detailed spectral analysis, they 
showed that a branch of eigenvalues of the 
indirect Kelvin-Voigt damping   ($d_{st}=0$ 
and
$d_{kv}>0$) system has a vertical asymptote
$\text{Re}\lambda=-\frac{\gamma^2}{2d_ {kv}}
$. This implies that the associated semigroup 
cannot be analytical and it is also not 
differentiable due to the distribution of 
the spectrum, even so the authors prove 
the exponential stability.   Many other 
papers were published in this direction, 
some of them can be viewed in 
\cite{Alabau1999,  Alabau-2002, Alabau-2011, 
Guglielmi2015, Hao-2015,  Renardy,Russell-1993, 
Tebou-2010,Tebou2012, Tebou-2017}.

Findings involving fractional dissipation 
were published by  Oquendo-Su\'a\-rez (2019) 
\cite{HPOquendo}. 
They studied the following abstract system:
\begin{eqnarray*}
\rho_1u_{tt}-\gamma_1 \Delta u_{tt}+
\beta_1 \Delta^2 u+\alpha v=0,  \quad 
x\in\Omega,\quad t>0,\\
\rho_2v_{tt}-\gamma_2 \Delta v_{tt}+
\beta_2 \Delta^2 v+ \alpha u+\kappa 
(-\Delta)^{\theta} v_t =0, \quad 
x\in\Omega,\quad t>0,
\end{eqnarray*}
 where
$\Omega$ be a bounded open set of
$\R^n$ with smooth boundary
$\partial\Omega$ and one of these equations 
is conservative and the other has fractional 
dissipative properties given by
$(-\Delta)^{\theta} v_t$, where
$0\leq\theta\leq 1$  
and where the coupling terms are
$\alpha u
$ and
$\alpha v$.  They showed that the  
semigroup decays polynomially with 
a rate that depends on
$\theta$ and some relations  among 
the structural coefficients of the system.     
Have shown that the rates obtained are optimal.   
Similarly, we can cite the investigations 
\cite{AShelTebou2021, Tebou-2020, KuangLiuSare2021, 
KuangLiuTebou2022, LiuYong1998, SareLiuRacke2019,
 Tebou-2013}.

This article is organized as follows: 
In Section \eqref{2.00}, we study the 
well-positioning of the 
\eqref{ISplacas-20}--\eqref{ISplacas-40} 
system through semigroup theory. 
We leave the main results for the 
last two sections. We divide the section 
\eqref{3.00} into three subsections, 
in the first one  \eqref{3.1} using 
spectral theory together with a semigroup 
theory characterization and frequency domain 
technique,  we demonstrate the lack of 
analyticity in the
$ R_L$,   
in the following Subsection \eqref{3.2} 
which was also subdivided into two 
subsubsections we determine two Gevrey 
classes of the semigroup
$S(t)$, in Subsubsection (3.2.1), 
we determine the Gevrey class
 $s_1 > \frac{1}{\Phi_{s_1}}=\frac{1}{2\max\{ \frac{1-\beta}{3-\beta}, \frac{\theta}{2+\theta-\beta}\}}$
in the region
$R_{G1} :=\{ (\theta,\beta)\in 
\mathbb{R}^2 / 2\leq \theta+ 2\beta, \;  
0<\theta \leq  1\;{\rm and}\; 
0<\beta<1\}$ 
and in the Subsubsection (3.2.2), 
we determine the Gevrey class
$s_ 2>\frac{1}{\phi_2}=\frac{2(2+\theta-\beta)}
{\theta}$ for the region
$R_{G2} :=\{ (\theta,  \beta)\in 
\mathbb{R}^2  /  3\theta+\beta\leq 1,  
0<\theta\leq 1\; {\rm and}\; 0<\beta < 1 \}$.  
We emphasize that
The Gevrey region of the semigroup S(t)  
is given by:
$R_G = R_{G1}\cup R_{G2}$.  
Since 
$R_{G1}\cap R_{G2}$ is not empty,  
in this intersection region the Gevrey 
class will be given by:
$s_i >\frac{1}{\max\{\Phi_{s_1}, \phi_2\}}$.  
Finally,  in the  Subsection \eqref{3.3},
 we prove that at the point
$(\theta,\beta) = (1, 1/2)$  the semigroup 
$S(t)$
 is analytic.  The investigation ends with 
 an observation
regarding the asymptotic behavior of
$S(t)$,  in which it is illustrated  
that the necessary estimates given in 
Propositions \eqref{limsup} and \eqref{EixoIm} 
imply that the semigroup associated with the 
system \eqref{ISplacas-20}--\eqref{ISplacas-40} 
is exponentially stable for
$(\theta,\beta) \in [0,1]\times (0,1]$.

\section{Well-Posedness of the System}
\label{2.00}
In this section,
 we will use the semigroup theory to 
 assure the existence and uniqueness of 
 strong solutions for the system 
 \eqref{ISplacas-20}-\eqref{ISplacas-40}  
 where the operator
$A$ is given by \eqref{Omenoslaplaciano}.  
Before that,  let us see some preliminary 
results.
It is important to recall that
$A$ is a positive self-adjoint operator 
with compact inverse in a complex Hilbert space
$D(A^0)=L^2(\Omega)$.  Therefore, the operator
$A^{\theta}$ is self-adjoint positive for all
$\theta\in\mathbb{R}$ and  bounded for
$\theta\leq 0$.   The embedding
$
D(A^{\theta_1})\hookrightarrow D(A^{\theta_2}),
$
is continuous for
$\theta_1>\theta_2$. Here,  the norm in
$D(A^{\theta})$ is given by
$\|u\|_{D(A^{\theta})}:=\|A^{\theta}u\|$,
$u\in D(A^{\theta})$, where
$\|\cdot\|$ denotes the norm in the Hilbert space
$L^2(\Omega)$.  Some of these spaces are:
$D(A^ { \frac 12 })=H_0^1(\Omega)$,
$D(A^0)=L^2(\Omega)$ and
$D(A^{-1/2})=H^{-1}(\Omega)$.  With this notation,
  for 
$\kappa$ positive we can extend the operator
$I+\kappa A^\beta$ in the following sense:
\begin{equation*}\label{BIsometrica}
(I+\kappa A^\beta)\colon D(A^{ \frac  \beta 2 })\to
 \;D(A^{ - \frac  \beta 2 })
\end{equation*}
defined by
\begin{equation}\label{EqPiDual}
\dual{( I+\kappa A^\beta)z_1}
{z_2}_{D(A^{ - \frac  \beta 2 })\times 
D(A^{ \frac  \beta 2 })}
=\dual{z_1}{z_2}+
\kappa\dual{A^{ \frac  \beta 2 }z_1}
{A^{ \frac  \beta 2 }z_2},
\end{equation}
for
$z_1,z_2\in D(A^{\beta /2})$, where
$\dual{\cdot}{\cdot}$ denotes the 
inner product in the Hilbert space
$D(A^0)$. Note that this operator is an 
isometric operator when we consider the 
equivalent norm in the space
$D(A^{ \frac  \beta 2 })$:
$$
\pbrack{\|z\|^2+\kappa\|A^{ \frac  \beta 2 } 
z\|^2}^ { \frac 12 }=
\|z\|_{D(A^\frac{\beta}{2})}^ { \frac 12 }.
$$
Now, we will use a semigroup approach to 
study the well-posedness of the system 
\eqref{ISplacas-20}-\eqref{ISplacas-40}.  
Taking
$w=u_t$,
$z=v_t$ and  considering
$U=( u,v,w,z)$ and
$U_0=(u_0,v_0,u_1,v_1)$, 
the system \eqref{ISplacas-20}--\eqref{ISplacas-40} 
can be written in the following abstract framework
\begin{equation}\label{Fabstrata}
    \frac{d}{dt}U(t)=\mathbb{B} U(t),
    \quad    U(0)=U_0,
\end{equation}
where the operator
$\mathbb{B}$ is given by
  \begin{gather} \label{operadorAgamma}
  \mathbb{B}U := \Big(w,\ z,\ 
  -(I+\kappa A^\beta)^{-1} (\alpha A^2 u
  +\gamma A z),\ - \alpha A v+\gamma 
  Aw-\delta A^{\theta} z \Big)
  \end{gather}
for
$U=(u,v,w,z)$. This operator will be 
defined in a suitable subspace of the 
phase space
 $          
\mathcal{H}:= D(A)\times D(A^\frac{1}{2})
\times D(A^\frac{\beta}{2})\times D(A^0).
 $
It is a Hilbert space with the inner product
\begin{align}  
\label{PISystem}    
\langle U_1,U_2\rangle_\mathcal{H}   := 
 \alpha  \dual{A u_1}{A u_2} & +\alpha 
 \dual{A^\frac{1}{2} v_1}{A^\frac{1}{2} v_2}
 \\
 &
 \notag
 + \dual{ w_{1}} {w_{2}}
+\kappa\dual{A^{ \frac  \beta 2 }w_1}
{A^{ \frac  \beta 2 }w_2}
+\dual{ z_{1}} {z_{2}}, 
\end{align}
for
$U_i=(u_i, v_i, w_i, z_i)\in \mathcal{H}$, 
$i=1,2$,  and induced norm
\begin{equation}\label{NormSystem}
\|U\|_ \mathcal{H}^2:=\alpha\|Au\|^2
+\alpha\|A^\frac{1}{2}v\|^2+\|w\|^2
+\kappa\|A^\frac{\beta}{2}w\|^2+\|z\|^2.
\end{equation}
In these conditions,  from
$D(A)\hookrightarrow D(A^{ \frac  \beta 2 })
\hookrightarrow D(A^0)$ 
and
$D(A^\frac{1}{2})\hookrightarrow
 D(A^{ - \frac  \beta 2 })$,
 we define the domain of
$\mathbb{B}$ as
\begin{align}
\label{dominioA}
   D(\mathbb{B}):   =
    \Big\{ U   \in \mathcal{H}\colon 
    & (w,z)\in  
    D(A)\times D(A^\frac {1}{2} ), \\
    & 
    (-\alpha A^2u-\gamma Az,  - \alpha 
    v- \delta A^{\theta-1} z )   
  \in (D(A^{ - \frac  \beta 2 }), 
    \  D(A)) \Big\}.  \notag
\end{align}
To show that the operator
$\mathbb{B}$ is the generator of a
$C_0-$semigroup,   we invoke a result f
rom Liu-Zheng \cite{LiuZ}.

\begin{theorem}[see Theorem 1.2.4 in \cite{LiuZ}] 
\label{TLiuZ}
Let
$\mathbb{A}$ be a linear operator with domain
$D(\mathbb{A})$ dense in a Hilbert space
$\mathcal{H}$. If
$\mathbb{A}$ is dissipative and
$0\in\rho(\mathbb{A})$, the resolvent set of
$\mathbb{A}$, then
$\mathbb{A}$ is the generator of a
$C_0$-semigroup of contractions on
$\mathcal{H}$.
\end{theorem}

Let us see that the operator
$\mathbb{B}$ satisfies the conditions 
of this theorem. Clearly, we see that
$D(\mathbb{B})$ is dense in
$\mathcal{H}$.   Performing the inner 
 product of
$\mathbb{B}U$ with
$U$, we have
\begin{equation}
\label{eqdissipative}
\text{Re}\dual{\mathbb{B}U}{U}_\mathcal{H}
=  -\delta\|A^{\theta/2} z\|^2, 
\quad\forall\ U\in D(\mathbb{B}),
\end{equation}
that is, the operator
$\mathbb{B}$ is dissipative. 

To complete the conditions of the 
above theorem,  it remains to
 be demonstrated that
$0\in\rho(\mathbb{B})$. Let
$F=(f_1,f_2,f_3,f_4)\in \mathcal{H}$, 
let us see that the stationary problem
$ \mathbb{B}U=F$ has a solution
$U=(u,v,w,z)$.  From the definition of 
the operator 
$\mathbb{B}$ given in
\eqref{operadorAgamma}, this system
can be written as
\begin{align}
w=f_1,\qquad& \quad\quad  
\alpha A^2 u=-\gamma Af_2-(I+\kappa A^\beta)
f_3 \label{exist-10}\\
z=f_2,\qquad &  \quad\quad \alpha Av =
\gamma Af_1-(I+\delta A^\theta) f_2+f_2-f_4.
 \label{exist-20}
\end{align}
This problem can be placed in a variational 
formulation: To find
$S=(u,v)$ such that
\begin{equation}\label{var-10}
b(S,Z)=h(Z):=\dual{h}{Z},\quad\forall\ Z
=(z_1,z_2)\in D(A)\times D(A^\frac{1}{2}).
\end{equation}
where\quad 
 $h=(-\gamma Af_2-(I+\kappa A^\beta)f_3,
 \gamma Af_1-\delta A^\theta f_2-f_4) \in 
 [D(A)\times D(A^0)]^\prime$\quad  and
\begin{equation*}
b(u,v; z_1,z_2):=\alpha\dual{Au}{Az_1}
+\alpha\dual{A^\frac{1}{2}v}{A^\frac{1}{2}z_2}.
\end{equation*}
The proof of the coercivity of this 
sesquilinear form
$b$ in Hilbert space
$D(A)\times D(A^\frac{1}{2})$ is  
immediate, now  applying the Lax-Milgram 
Theorem and taking into account the 
first equations of 
\eqref{exist-10}-\eqref{exist-20},
 we have a unique solution
$U\in \mathcal{H}$.   
Since  this solution satisfies the 
system (\ref{exist-10})-(\ref{exist-20}) 
in a weak sense, from these equations,
 we can conclude that
$U\in D(\mathbb{B})$.  Finally,  for
$t=(u,v)$,  we have
\begin{equation*}
b(t,t)=\alpha\|Au\|^2+\alpha\|A^\frac{1}{2}v\|^2.
\end{equation*}
From second  equations of  
\eqref{exist-10}-\eqref{exist-20},    
applying Cauchy-Schwarz and Young 
inequalities to the second member 
of there exists
$C>0$ such that
\begin{equation*}
\alpha\|A u\|^2+ \alpha\|A^\frac{1}{2} 
v\|^2\leq C\|F\|^2_\mathcal{H}.
\end{equation*}
This inequality and the first equations of 
(\ref{exist-10})-(\ref{exist-20}) imply that
$\|U\|_\mathcal{H}\leq C\|F\|_\mathcal{H}$, 
then 
$0$ belongs to the resolvent set
$\rho(\mathbb{B})$. Consequently, 
from Theorem \ref{TLiuZ}  we have 
$\mathbb{B}$ as the generator of a 
contractions semigroup.  
As a consequence of the above 
Theorem\ref{TLiuZ} we have
\begin{theorem}
Given
$U_0\in\mathcal{H}$,  
there exists a unique weak solution
$U$ to  the problem \eqref{Fabstrata} 
satisfying 
$ 
U\in C([0, +\infty), \mathcal{H}).
$ 
Futhermore, if
$U_0\in  \hbox{D}(\mathbb{B}^k), \;
 k\in\mathbb{N}$, then the solution
$U$ of \eqref{Fabstrata} satisfies
$$
U\in \bigcap_{j=0}^kC^{k-j}([0,+\infty), 
 \hbox{D}(\mathbb{B}^j).
$$
\end{theorem}

It follows that 
$C$ and
$C_\tau$ will denote a positive 
constant that assumes different 
values in different places and  
the coupling coefficient
$\gamma$  is  assumed positive. 
The results remain valid when this 
coefficient is negative.

\section{ Regularity Results}
\label{3.00}

In this section,
 we discuss the regularity of the semigroup
$S(t)=e^{\mathbb{B}t}$  in three 
subsections: First we analyze the 
lack of analyticity of
$S(t)$ in the region
$R_L:=\{ (\theta,\beta)/  
0\leq\theta\leq 1\; \rm{and}\; 0<\beta\leq 1\}
-\{(\frac{1}{2},1)\}$,    
then we study the Gevrey class of
$S(t)$ in the region
$R_G:=\{(\theta,\beta)\in (0,1]\times(0,1]\}$ 
and we end this section by 
showing the analyticity of
$S(t)$ at the point
$(\theta,\beta)=(\frac{1}{2} ,1)$.

For 
$\lambda\in  \R$ and
$F=(f_1,f_2,f_3,f_4)\in \mathcal{H}$  
the solution
$U=(u,v,w,z)\in\hbox{D}(\mathbb{B
 })$ of the stationary system
$(i\lambda I- \mathbb{B
 })U=F$ can be written in the form  
\begin{eqnarray}
i\lambda u-w &=& f_1\label{esp-10}\\
i\lambda v-z &=& f_2\label{esp-20}\\
i\lambda (I+\kappa A^\beta) w+ \alpha A^2 u
+\gamma Az &=&(I+\kappa A^\beta)f_3\label{esp-30}\\
i\lambda  z+\alpha A v-\gamma Aw+\delta 
A^{\theta} z&=& f_4.\label{esp-40}
\end{eqnarray}
We have
\begin{equation}\label{dis-10}
\delta\|A^{\frac{\theta}{2}}z\|^2=
\text{Re}\dual{(i\lambda I -\mathbb{B})U}
{U}_\mathcal{H}=\text{Re}\dual{F}{U}_\mathcal{H}\leq
 \|F\|_\mathcal{H}\|U\|_\mathcal{H}.
\end{equation}
From equations (\ref{esp-20}) and \eqref{dis-10},
 we have
\begin{eqnarray}\label{dis-10A}
\lambda^2\|A^{\frac{\theta}{2}}v\|^2 
&\leq& 
C\key{\|F\|_\mathcal{H}
\|U\|_\mathcal{H}+\|F\|^2_\mathcal{H}}.
\end{eqnarray} 
As
$\frac{\theta-2}{2}\leq  0\leq\frac{\theta}{2}$, 
taking into account the continuous embedding
$D(A^{\theta_2})\hookrightarrow D(A^{\theta_1})$,
$\theta_2>\theta_1$ and \eqref{dis-10}, 
we obtain
\begin{eqnarray}
\label{dis-10BC}
  \|A^\frac{\theta-2}{2}z\|^2 \leq 
  C\|F\|_\mathcal{H}\|U\|_\mathcal{H}
  \qquad{\rm and}\qquad
  \|z\|^2 \leq 
  C\|F\|_\mathcal{H}\|U\|_\mathcal{H}.
\end{eqnarray}

The following theorem characterizes
 the analyticity of
$S(t)$:

\begin{theorem}[see \cite{LiuZ}]
\label{LiuZAnaliticity}
    Let
$S(t)=e^{\mathbb{B}t}$ be
$C_0$-semigroup of contractions  
on a Hilbert space.  
Suppose that
    \begin{equation*}
    \rho(\mathbb{B})
    \supseteq\{ i\lambda; \; \lambda\in \R \} 
     \equiv i\R
    \end{equation*}
     Then
$S(t)$ is analytic if and only if
    \begin{equation}\label{Analiticity}
     \limsup\limits_{|\lambda|\to
        \infty}
    \|\lambda(i\lambda I-\mathbb{B})^{-1}
    \|_{\mathcal{L}(\mathcal{H})}<\infty
    \end{equation}
    holds.
    
\end{theorem}
 
\begin{remark}\label{ObsEquivAnaly}  
\rm
To show the \eqref{Analiticity} condition,  
it suffices to show that, given
$\tau>0$ there exists a constant
$C_\tau > 0$ such that the solutions of 
\eqref{esp-10}--\eqref{esp-40}, 
for
$|\lambda|>\tau$ satisfy the inequality
\begin{align}
\label{EquivAnaliticity}
& 
|\lambda|\|U\|_\mathcal{H}\leq 
C_\tau\|F\|_\mathcal{H} \ \
 \Longleftrightarrow  \\
 &
 \notag
 |\lambda|\|U\|_\mathcal{H}^2=
|\lambda|\|(i\lambda I-\mathbb{B})^{-1}
F\|_\mathcal{H}^2\leq 
C_\tau\|F\|_\mathcal{H}\|U\|_\mathcal{H}.
\end{align}
\end{remark}

Next, we demonstrate two lemmas that 
are essential to achieve the results 
of  following subsections results.

\begin{proposition}  \label{limsup} Let
$0\leq\theta\leq 1$ and
$0<\beta\leq 1$. 
The operator
$\mathbb{B}$ satisfies the following resolvent
estimate
\begin{equation}\label{ExponentialP}
 \limsup\limits_{|\lambda|\to
   \infty}   \|(i\lambda I-\mathbb{B})^{-1}
   \|_{\mathcal{L}(\mathcal{H})}<\infty.
\end{equation}
\end{proposition}

\begin{proof}[\bf Proof.]
Similarly to observation \eqref{ObsEquivAnaly},    
we will show
\begin{align}
\label{ExponentialP1}
&
\|U\|^2_\mathcal{H}
\leq C_\tau\|F\|_\mathcal{H}\|U\|_\mathcal{H} 
\qquad   \Longleftrightarrow \\
&
\notag
 \|U\|^2_\mathcal{H}=\alpha\|Au\|^2+
\alpha\|A^\frac{1}{2}v\|^2+\|w\|^2+
\kappa\|A^\frac{\beta}{2}w\|^2+\|z\|^2\leq 
C_\tau\|F\|_\mathcal{H}\|U\|_\mathcal{H}.
\end{align}
 
 \smallskip
 
 \noindent
{\bf  Estimate of terms: $\alpha\|Au\|^2$   
and $\|w\|^2+\kappa\|A^\frac{\beta}{2}w\|^2$.}
Taking the duality product between  
equation (\ref{esp-30}) and $u$ and 
using the equation (\ref{esp-10}). 
Taking advantage 
of the self-adjointness of the powers 
of the operator $A$, we obtain  
\begin{align}
\nonumber
 \alpha\|Au\|^2 
& =  -\gamma\dual{z}{Au}+\dual{w}{i\lambda u}
+\kappa\dual{A^\beta w}{i\lambda u}
+\langle f_3, u\rangle+\kappa\dual
{A^\frac{\beta}{2} f_3}{A^\frac{\beta}{2} u}\\
\label{Exp1010AU}
& =  -\gamma\dual{z}{Au}+
\|w\|^2+\kappa\|A^\frac{\beta}{2}w\|^2+\dual{w}
{f_1} \\
&
\notag
+\kappa
\dual{A^\frac{\beta}{2}w}{A^\frac{\beta}{2}f_1}  
+\dual{f_3}{u} +\kappa\dual{A^\frac{\beta}{2} 
f_3}{A^\frac{\beta}{2} u}.
\end{align}

On the other hand,  performing the 
duality product of \eqref{esp-40}  and   
$(I+\kappa A^\beta)A^{-1}w$,   
using the identities \eqref{esp-10} and  
\eqref{esp-20} and  in addition to take advantage  
of the self-adjointness of the powers of 
the operator
$A$,  we obtain
\begin{align*}
\gamma\kappa\|A^\frac{\beta}{2}w\|^2  
& =  -\gamma\|w\|^2-
\dual{A^{-1}z}{i\lambda(I+\kappa A^\beta)w}+
\alpha\dual{A^\frac{\theta}{2}v}
{A^{-\frac{\theta}{2}}(i\lambda u-f_1)}\\
&   +\alpha\kappa\dual
{A^\frac{\theta}{2}v}{A^{\beta-\frac{\theta}{2}}
(i\lambda u-f_1)}+\delta\dual{A^\frac{\theta}
{2}z}{(I+\kappa A^\beta)A^{\frac{\theta}{2}-1}w} \\
&
-\dual{f_4}{(I+\kappa A^\beta)A^{-1}w}.
\end{align*}
Therefore,
\begin{align}
\nonumber
\gamma\kappa\|A^\frac{\beta}{2}w\|^2  
& =  
 -\dual{A^{-1}z}
{i\lambda(I+\kappa A^\beta)w}-\alpha\dual
{A^\frac{\theta}{2}(z+f_2)}{A^{-\frac{\theta}{2}}u}
 \\
 \nonumber
 &
 -\alpha\dual{A^\frac{\theta}{2}v}
{A^{-\frac{\theta}{2}} f_1} 
-\alpha\kappa\dual{A^\frac{\theta}{2}(z+f_2)}
{A^{\beta-\frac{\theta}{2}} u} \\
 \nonumber
 & 
 -\alpha\kappa\dual
{A^\frac{\theta}{2}v}{A^{\beta-\frac{\theta}{2}}f_1} 
+\delta\dual{A^\frac{\theta}{2}z}
{A^{\frac{\theta}{2}-1}w}\\
\label{Eq01EstAbeta2}
&  
+\delta\kappa\dual{A^\frac{\theta}{2}z}
{A^{\beta+\frac{\theta}{2}-1}w}-\dual{f_4}
{(I+\kappa A^\beta)A^{-1}w}
-\gamma\|w\|^2.
\end{align}
From \eqref{esp-30}, we have
$$
i\lambda(I+\kappa A^\beta)w=-\alpha A^2u
-\gamma Az+(I+\kappa A^\beta)f_3 , 
$$ 
and estimate \eqref{dis-10},    from
$\varepsilon>0$, exist
$C_\varepsilon>0$ such that 
\begin{align}
\nonumber
&
|\dual{A^{-1}z}{i\lambda(I+\kappa A^\beta)w}| \\
\notag 
&=  |\dual{A^{-1}z}{-\alpha A^2u-
\gamma Az+(I+\kappa A^\beta)f_3}|\\
\nonumber
& =  |-\alpha\dual{A^\frac{\theta}{2}z}
{A^{1-\frac{\theta}{2}}u}-\gamma\|z\|^2+\dual{z}
{(I+\kappa A^\beta)A^{-1}f_3}|\\
\label{Eq02EstAbeta2}
&\leq  C_\varepsilon\|A^\frac{\theta}{2}z\|^2
+\varepsilon\|A^{1-\frac{\theta}{2}}u\|^2
+C\|F\|_\mathcal{H}\|U\|_\mathcal{H}
\end{align}
Now,  using estimate \eqref{Eq02EstAbeta2} 
in \eqref{Eq01EstAbeta2},    
taking account estimate  \eqref{dis-10},  
applying Cauchy-Schwarz, Young inequalities, 
we have
\begin{align}
\nonumber
&
\gamma\kappa\|A^\frac{\beta}{2}w\|^2 
+\gamma\|w\|^2  \\
\notag
& \leq   
C_\varepsilon\|A^\frac{\theta}{2}z\|^2
+\varepsilon\|A^{1-\frac{\theta}{2}}u\|^2+
C\|F\|_\mathcal{H}\|U\|_\mathcal{H}+
C|\dual{A^\frac{\theta}{2} z}
{A^{-\frac{\theta}{2}}u}|\\
\nonumber
&  
  +C|\dual{A^\frac{\theta}{2}f_2}
{A^{-\frac{\theta}{2}}u}|+C\|A^\frac{1}{2}
v\|\|A^{-\frac{1}{2}}f_1\| +
C|\dual{A^\frac{\theta}{2}z}
{A^{\beta-\frac{\theta}{2}}u}|\\
\nonumber
&    +
C\|A^\frac{1}{2}f_2\|\|A^{\beta-\frac{1}
{2}}u\|+C\|A^\frac{1}{2}v\| 
\|A^{\beta-\frac{1}{2}}f_1\|+
C|\dual{A^\frac{\theta}{2}z}
{A^{\frac{\theta}{2}-1}w}|\\
\label{Eq03EstAbeta2}
&  +C|\dual
{A^\frac{\theta}{2}z}{A^{\beta+\frac{\theta}{2}-1}w}
|+C\|f_4\|\|A^{-1}w\|+ C\|f_4\|\|A^{\beta-1}w\|.
\end{align}
 From
$-\frac{\theta}{2}\leq\beta-\frac{\theta}{2}
\leq 1-\frac{\theta}{2}$,
$\beta-\frac{1}{2}\leq 1$ and
$\frac{\theta}{2}-1\leq\beta+\frac{\theta}{2}
-1\leq\frac{\beta}{2}$,  
taking into account the continuous embedding  
$D(A^{\theta_2}) \hookrightarrow 
D(A^{\theta_1}),\;\theta_2>\theta_1$,  
$\beta\leq1$ we have,    for 
$\varepsilon>0$, there exists a 
positive constant
$C_\varepsilon$, independent of
$\lambda$ such that
\begin{align}
\label{Eq04EstAbeta2}
\|w\|^2+\kappa\|A^\frac{\beta}{2}w\|^2  
  \leq   \varepsilon\|A^{1-\frac{\theta}
{2}}u\|^2+C\|F\|_\mathcal{H}\|U\|_\mathcal{H}.
\end{align}
Now, from estimate \eqref{Eq04EstAbeta2} 
in \eqref{Exp1010AU} and using 
Cauchy-Schwarz and Young inequalities,
 we have
\begin{align*}
\|Au\|^2 
& \leq   C|\dual{A^\frac{\theta}{2}z}
{A^{1-\frac{\theta}
{2}}u}|+C|\dual{A^\frac{\beta}{2}w}
{A^\frac{\beta}{2}f_1}+\dual{f_3}{u} 
+\dual{A^\frac{\beta}{2} f_3}
{A^\frac{\beta}{2} u}|   \\
&  +C|\dual{w}{f_1}|+
\varepsilon\|A^{1-\frac{\theta}{2}}u\|^2
+C\|F\|_\mathcal{H}\|U\|_\mathcal{H}.
\end{align*}
Then, from
$1-\frac{\theta}{2}\leq 1$, 
taking into account the continuous embedding  
$D(A^{\theta_2}) \hookrightarrow D(A^{\theta_1}),
\;\theta_2>\theta_1$, we have
\begin{equation}\label{Eq05EstAbeta2}
\|Au\|^2\leq C\|F\|_\mathcal{H}\|U\|_\mathcal{H}.
\end{equation}
Finally, using  \eqref{Eq05EstAbeta2} in 
\eqref{Eq04EstAbeta2}, obtain
\begin{equation}\label{Eq06EstAbeta2}
\|w\|^2+\kappa\|A^\frac{\beta}{2}w\|^2 
\leq C\|F\|_\mathcal{H}\|U\|_\mathcal{H}.
\end{equation}

 \smallskip
 
 \noindent
{\bf Estimate of term 
$\alpha\|A^\frac{1}{2}v\|^2$.}
Similarly, applying the duality product 
between the  equation (\ref{esp-40}) with 
$v$,  
using the equation (\ref{esp-20}) and in 
addition to the self-adjointness of the 
powers of the operator $A$ to rewrite 
some terms. Then, we have
\begin{align}
\notag
\alpha\|A^\frac{1}{2}v\|^2 
&=  \gamma\dual{Aw}{v}+\|z\|^2
-\delta \dual{A^\frac{\theta}{2}z}
{A^\frac{\theta}{2}v}+\langle z, 
f_2\rangle+\dual{f_4}{v}\\
 \nonumber
 &=   -\gamma\dual{A^{1-\frac{\theta}
 {2}}u}{A^\frac{\theta}{2}z}-
 \gamma\dual{Au}{f_2}-\gamma\dual{Af_1}{v} 
 \\
 & 
 -\delta \dual{A^\frac{\theta}{2}
 (i\lambda v-f_2)}{A^\frac{\theta}{2}v} 
\label{Exp1012}
   +\|z\|^2+\langle z, f_2\rangle+
 \langle f_4, v\rangle,
\end{align}
now,  as 
$$
|-\gamma\dual{A^{1-\frac{\theta}{2}}u}
{A^\frac{\theta}{2}z}|\leq 
\varepsilon\|A^{1-\frac{\theta}{2}}u\|^2
+C_\varepsilon\|A^\frac{\theta}{2}z\|^2
$$   
and for $\varepsilon>0$ exists 
$C_\varepsilon>0$, 
such that  
$$
 |-\delta\dual{A^\frac{\theta}{2}z}
{A^\frac{\theta}{2}v}|\leq 
C_\varepsilon\|A^\frac{\theta}{2}z\|^2
+\varepsilon\|{A^\frac{\theta}{2}}v\|^2\leq   
C_\varepsilon\|A^\frac{\theta}{2}z\|^2
+\varepsilon\|{A^\frac{1}{2}}v\|^2 , 
$$    
using the estimative 
\eqref{Eq05EstAbeta2} 
and applying  Cauchy-Schwarz inequality and 
Young inequalities  and continuous embedding for 
$\theta-\frac{1}{2}\leq\frac{1}{2}$  in 
\eqref{Exp1012}, we have the inequality
\begin{equation}
\label{Exp1016}
\alpha\|A^\frac{1}{2}v\|^2 \leq 
C \|F\|_\mathcal{H}\|U\|_\mathcal{H}
\ \hbox{ for }  0\leq\theta\leq 
1 \text{ and }\   0<\beta\leq 1.
\end{equation}

 \smallskip
 
 \noindent
{\bf   Estimate of $\|U\|^2$.} Therefore,  estimates 
\eqref{dis-10BC}$_{(2)}$, \eqref{Eq05EstAbeta2}, 
\eqref{Eq06EstAbeta2}  and  \eqref{Exp1016}, 
one easily derives 
\begin{equation*}
\|U\|^2_\mathcal{H}\leq 
C\|F\|_\mathcal{H}\|U\|_\mathcal{H}
 \text{ for }\   0\leq\theta\leq 
1  \text{and }  0<\beta\leq 1.
\end{equation*}
Then, as $U=(i\lambda I-\mathbb{B})^{-1}F$,  
we obtain
\begin{equation}\label{Exp1017}
\sup\limits_{F\in\mathcal{H},}   
\dfrac{\|(i\lambda I-\mathbb{B})^{-1}
F\|_\mathcal{H}}{\|F\|_\mathcal{H}}\leq     
\sup\limits_{F\in\mathcal{H}}   
C \text{  for }  0\leq\theta\leq 
1 \text{ and }  0<\beta\leq 1.
\end{equation}
Thus,   we conclude the proof of this proposition.
\end{proof}

\begin{proposition} \label{EixoIm}
\label{iR}
Let
$\varrho(\mathbb{B})$ be the resolvent set of operator
$\mathbb{B}$. Then
\begin{equation}
i \mathbb{R}\subset\varrho(\mathbb{B}).
\end{equation}
\end{proposition}
 
\begin{proof}[\bf Proof.]
To prove
$i\R\subset\rho(\mathbb{B})$ we will 
argue by reductio ad absurdum.  Suppose that
$i\R\not\subset \rho(\mathbb{B})$.  
As
$0\in\rho(\mathbb{B})$ and
$\rho(\mathbb{B})$ is open, we consider 
the highest positive number
$\lambda_0$ such that the
$ ( -i\lambda_0,i\lambda_0) \subset\rho(\mathbb{B})$,
 then
$i\lambda_0$ or
$-i\lambda_0$ is an element of the spectrum
$\sigma(\mathbb{B})$. We Suppose
$i\lambda_0\in \sigma(\mathbb{B})$ 
(if
$-i\lambda_0\in \sigma(\mathbb{B})$ 
the proceeding is similar). Then, for
$0<\delta<\lambda_0$ there exists  
a sequence of real numbers
$(\lambda_n)$, with
$\delta\leq\lambda_n<\lambda_0$,
$\lambda_n\con \lambda_0$, and a 
vector sequence 
$U_n=(u_n,v_n,w_n,z_n)\in D(\mathbb{B})$ 
with  unitary norms, such that
\begin{equation*}
\|(i\lambda_nI-\mathbb{B}) 
U_n\|_\mathcal{H}=\|F_n\|_\mathcal{H}\con 0,
\end{equation*}
as
$n\con \infty$.  From  the estimates  
\eqref{Eq05EstAbeta2}   and  \eqref{Exp1016} 
for
$0\leq\theta\leq 1$ and
$0<\beta\leq 1$, we have
\begin{equation*}
\alpha \|A u_n\|^2 \leq 
C\|F_n\|_\mathcal{H}
\|U_n\|_\mathcal{H}\to 0   \text{ and }  
 \alpha\|A^ { \frac 12 }v_n\|^2 \leq  
C\|F_n\|_\mathcal{H}\|U_n\|_\mathcal{H}\to 0.
\end{equation*}
In addition to  the estimates  
\eqref{dis-10BC}$_{(2)}$ and  
\eqref{Eq06EstAbeta2} for
$0\leq\theta\leq 1$ and
$0<\beta\leq 1$, we have
\begin{equation*}
\kappa\|A^\frac{\beta}{2}w_n\|^2 
+\|w_n\|^2+\|z_n\|^2 \con 0.
\end{equation*}
Consequently,  
\begin{equation*}
\alpha\|Au_n\|^2 +\alpha\|A^ { \frac 12 }v_n\|^2
+\|w_n\|^2+\kappa\|A^\frac{\beta}{2}w_n\|^2
+\|z_n\|^2 \con 0.
\end{equation*}
Therefore, we have 
$\|U_n\|_\mathcal{H}\con 0$ but this is 
absurd, since
$\|U_n\|_\mathcal{H}=1$ for all
$n\in\N$. Thus,
$i\R\subset \rho(\mathbb{B})$.  
Thereby,   
we conclude the proof of this proposition.
\end{proof}
 
\subsection{$S(t)=e^{\mathbb{B}t}$ 
is not analytic 
in the  region
$R_L$}
\label{3.1}
In this subsection, using 
Theorem \ref{LiuZAnaliticity}, 
 we will show that
$S(t)$ is not analytic in the  region
$R_L$.
\tikzstyle{polygon}=[very thick]
\tikzstyle{point}=[ultra thick,draw=gray,cross out,inner 
sep=0pt,minimum width=4pt,minimum height=4pt,]
\tikzstyle{line}=[red]
\begin{center}
    \begin{tikzpicture}[scale=3, rotate=0]
    \label{Figura01}
 \foreach \x in {0,1}
    \draw (\x cm,1pt) -- (\x cm,-1pt) 
    node[anchor=north] {$\x$};
    \foreach \y in {0,1}
    \draw[thick,->] (0,0) -- (1.25,0)
    node[anchor=north west] {$\theta$ axis};
    \draw[thick,->] (0,0) -- (0,1.25) 
    node[anchor=south east] {$\beta$ axis};
     \draw (0.5, 1.12) node[anchor=center] 
     {$( \frac{1}{2},1)$};
     \draw (-0.1, 1) node[anchor=center] {$1$};
\coordinate[label=left:] (a) at (0,0);
    \coordinate[label=right:] (b) at (1,0);
    \coordinate[label=right:] (c) at (1,1);
    \coordinate[label=left:] (d) at (0,1);

    \draw[rectangle, fill=black!20!white] 
    (a) -- (b) -- (c) --  (d) -- (a) -- cycle;
    \draw[black!100!black, ultra thick]
     (0, 0) -- (0, 1);
    \draw[black!100!black, ultra thick]
     (0, 1) -- (1, 1);
   \draw[white,  dashed] (0, 0) -- (1, 0);
   \draw[black!65!black, ultra thick] 
   (1, 0) -- (1, 1);
     \draw[fill=white!95!black, dashed] 
     (1,0) circle (0.02);
   \draw[fill=white!95!black, dashed]
   (0,0) circle (0.02);
    \draw[fill=white!95!black, dashed] 
    (1/2,1) circle (0.02);
 \draw[fill=black!100!white, dashed] 
 (0,1) circle (0.025);
 \draw[fill=black!100!white, dashed] 
 (1,1) circle (0.025);
    \end{tikzpicture}

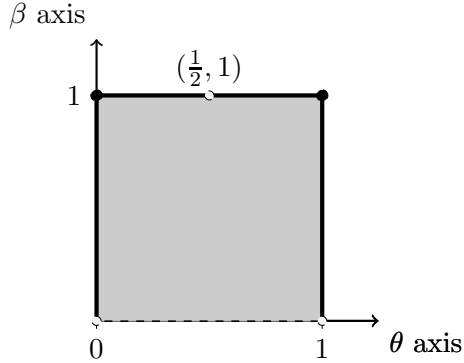
    {\bf Fig.  01}: Region
$R_L$ of lack of Analyticity.
\end{center}
 In the proof, we will use the 
 \eqref{EquivAnaliticity} equivalence of 
 the theorem \eqref{LiuZAnaliticity}, 
 note that non-verification of identity 
 \eqref{EquivAnaliticity} implies the lack 
 of analyticity of the associated semigroup
$S(t)=e^{\mathbb{B}t}$.

\begin{theorem}
\label{lack KV and structural} Let
$S(t)=e^{t\mathbb{B}}$ be the
$C_{0}
$-semigroup of 
contractions over the Hilbert space
$\mathcal{H}$ associated
with the system 
\eqref{ISplacas-20}-\eqref{ISplacas-40} 
that is not analytic when
$(\theta,\beta) \in R_L$.
\end{theorem}

\begin{proof}[\bf Proof.]
 Let us construct a sequence
$F_n$ such that the solutions of 
\begin{equation*}
i\lambda_n U_n-\mathbb{B}U_n=F_n,
\end{equation*}
satisfy
$|\lambda_n|\|U_n\|_\mathcal{H}\to\infty$  in
$R$,  which in turn implies
\begin{equation}\label{LackAnaliticity}
\|\lambda_n(i\lambda_nI-\mathbb{B})^{-1}
F_n\|_\mathcal{H}\to\infty \ 
{\rm for  }\  (\theta,\beta) 
\in\left[  0, 1\right] 
\times (0,1]-\{(\tfrac{1}{2},1)\}.
\end{equation}
which means that the corresponding 
semigroup is not analytic.

From here, in order to prove 
\eqref{LackAnaliticity}, we divide 
our analysis into four cases.  
For all of them case we will 
choose the function
$F_n=(0,0,-e_n,0)\in\mathcal{H}$.
 
As the   operator
$A$ defined in \eqref{Omenoslaplaciano}  
is positive, self-adjoint and it has 
compact resolvent, its spectrum  is 
constituted by positive eigenvalues
$(\sigma_n)$ such that
$\sigma_n\con \infty$ as
$n\con \infty$. For
$n\in \N$,  we denote with
$e_n$   an unitary
$D(A^\frac{\beta}{2})$-norm eigenvector
 associated to the eigenvalue
$\sigma_n$, that is,
\begin{equation}\label{auto-10A}
A^\varphi e_n=\sigma_n^\varphi e_n,\quad 
\|e_n\|_{D(A^\frac{\beta}{2})}=1,
\quad n\in\N, \; \varphi\in\mathbb{R}.
\end{equation}

Let $F_n=(0,0,-e_n,0)\in \mathcal{H}$.
The solution 
$U_n=(u_n,v_n,w_n,z_n)$ of the system 
$(i\lambda_n I-\mathbb{B})U_n=F_n$ satisfies 
$w_n=i\lambda_n u_n$, $z_n=i\lambda_n v_n$ 
and the following equations
\begin{align*}
\lambda^2_n(I+\kappa A^\beta) u_n-\alpha 
A^2u_n-i\lambda_n\gamma Av_n&=  e_n+\kappa 
A^\beta e_n,\\
\lambda^2_n  v-\alpha Av_n+i\lambda_n\gamma Au_n
    -i\lambda_n\delta A^{\theta} v_n&= 0.
\end{align*}
Let us see whether this system admits 
solutions of the form
$
u_n=\mu_n e_n,
$
$ v_n=\nu_n e_n,
$
for some complex numbers 
$\mu_n$ and 
$\nu_n$. Then, the numbers 
$\mu_n$, 
$\nu_n$ should satisfy the algebraic system
\begin{align}
\label{eq001systemA}
\big\{\lambda^2_n+\kappa\lambda^2_n\sigma_n^\beta
- \alpha\sigma_n^2\big\}\mu_n-
i\lambda_n\gamma\sigma_n\nu_n 
&=  1+\kappa\sigma_n^\beta,\\
\label{eq002systemA}
i\lambda_n\gamma\sigma_n\mu_n+\big\{\lambda^2_n
-\alpha\sigma_n-i\delta\lambda_n
\sigma_n^\theta \big\}\nu_n&=  0.
\end{align}
On the other hand,  
solving the system 
\eqref{eq001systemA}-\eqref{eq002systemA}, 
we find out that
\begin{eqnarray}\label{Muexpre01A}
\mu_n=\frac{(1+\kappa\sigma_n^\beta) 
 \big\{p_{2,n}(\lambda^2_n)
    -i\delta\lambda_n\sigma_n^\theta\big\}}
{p_{1,n}(\lambda^2_n)p_{2,n}(\lambda^2_n)
-\gamma^2\lambda^2_n\sigma_n^2
    -i\delta\lambda_n\sigma_n^\theta 
     p_{1,n}(\lambda^2_n)}
\end{eqnarray}
and
\begin{eqnarray}\label{Nu001}
\nu_n=\dfrac{-i\gamma\lambda_n
\sigma_n(1+\kappa\sigma_n^\beta)}
{p_{1,n}(\lambda^2_n)p_{2,n}(\lambda^2_n)
-\gamma^2\lambda^2_n\sigma_n^2
    -i\delta\lambda_n\sigma_n^\theta 
     p_{1,n}(\lambda^2_n)},
\end{eqnarray}
where
\begin{eqnarray}\label{P1P2A}
p_{1,n}(\lambda^2_n):=
\lambda^2_n+\kappa\lambda^2_n\sigma_n^\beta
-\alpha\sigma_n^2\quad\text{and}\quad 
p_{2,n}(\lambda^2_n)=\lambda^2_n-\alpha\sigma_n.
\end{eqnarray}
Taking 
$s_n=\lambda_n^2$ and considering  
the polynomial in $s_n$
\begin{align*}
q_n(s_n)   :   & =   p_{1,n}(s_n)p_{2,n}(s_n)
-\gamma^2\sigma_n^2s_n\\
& =  \underbrace{(1+\kappa\sigma_n^\beta)}_as_n^2
\underbrace{-[\alpha(\sigma_n
+\kappa\sigma_n^{1+\beta})+
(\alpha+\gamma^2)\sigma_n^2]}_bs_n
+\underbrace{\alpha^2\sigma_n^3}_c.
\end{align*}
We will now divide our analysis into four cases:

 \smallskip
 
 \noindent
{\bf First case:}     
\begin{center}
    \begin{tikzpicture}[scale=3, rotate=0]
    \label{Figura02}
    \draw[thick,->] (0,0) -- (1.25,0);
    \draw[thick,->] (0,0) -- (0,1.25);
    \foreach \x in {0,1}
    \draw (\x cm,1pt) -- (\x cm,-1pt) node[anchor=north] {$\x$};
    \foreach \y in {0,1}
   \draw[thick,->] (0,0) -- (1.25,0)
    node[anchor=north west] {$\theta$ axis};
    \draw[thick,->] (0,0) -- (0,1.25) node[anchor=south east] {$\beta$ axis};
     \draw (1.07,0.54) node[anchor=center] {$\qquad(1, \frac{1}{2})$};
       \draw (0.4, 1.12) node[anchor=center] {$\qquad( \frac{1}{2},1)$};
          \draw (-0.1, 1) node[anchor=center] {$1$};
 \coordinate[label=left:] (a) at (0,1);
    \coordinate[label=right:] (b) at (1/2,1);
    \coordinate[label=right:] (c) at (1,1/2);
    \coordinate[label=left:] (d) at (1,0);
    \coordinate[label=left:] (e) at (0,0);
  \draw[rectangle, fill=black!30!white] (a)  -- (b) --  (c) -- (d) -- (e) -- (a) -- cycle;    

  \draw[black!95!black, ultra thick] (1,0) -- (1, 1/2);
   \draw[black!95!black, ultra thick] (0,0) -- (0,1);
   \draw[black!95!black, ultra thick, dashed] (0,1) -- (1/2, 1);
 \draw[black!5!white, ultra thick, dashed] (0,0) -- (1, 0);
  \draw[black!5!black, ultra thick] (1/2,1) -- (1, 1/2);
  \draw[white!5!white, ultra thick] (1, 1) -- (1,1/2);   
    \draw[white!05!white, ultra thick] (1/2, 1) -- (1, 1);
  \draw[fill=white!100!white, dashed] (0,1) circle (0.025);
    \draw[fill=white!100!white,dashed] (1/2,1) circle (0.025);
   \draw[fill=black!100!white, dashed] (1,1/2) circle (0.025);
    \draw[fill=white!100!white dashed] (1,0) circle (0.025);
   \draw[fill=white!100!white, dashed] (0,0) circle (0.025);
    \end{tikzpicture}
    
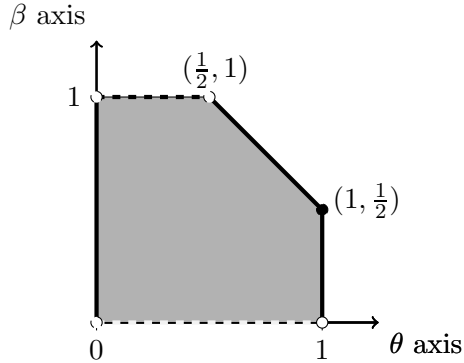
    {\bf Fig.  02}: Region 
    $R_1:=\{(\theta,\beta)\in [0,1]\times(0,1)/ \frac{3}{2}
    \geq \theta+\beta\} $.
\end{center}

Considering initially the region 
$$
R_{11}=\{ (\theta,\beta)\in \R^2 / 
0\leq \theta\leq 1 \text{  and }
\  0<\beta<1  \}
$$  
and choosing  the negative root of 
$q_n(\lambda_n^2)=0$,  we have
\begin{align}
\nonumber
|\lambda_n|^2 
 &=
 \Big |\dfrac{(-b-\sqrt{b^2-4ac})(-b+\sqrt{b^2-4ac})}
{2a(-b+\sqrt{b^2-4ac})} \Big |\\
\label{OrdemLambdaCase4}
&= \Big |\dfrac{2c}
{ (-b+\sqrt{b^2-4ac})} \Big |\approx  |\sigma_n|.
\end{align}

From \eqref{P1P2A},  
we have
$$
p_{1,n}(\lambda_n^2)-\kappa\lambda_n^2
\sigma_n^\beta+\alpha\sigma_n^2=p_{2,n}
(\lambda_n^2)+\alpha\sigma_n
$$
then
\begin{align*}
p^2_{1,n}(\lambda_n^2)-[\kappa \lambda_n^2\sigma_n^\beta 
+\alpha\sigma_n-\alpha\sigma_n^2]p_{1,n}(\lambda_n^2)
-\gamma^2\lambda_n^2\sigma_n^2=0.
\end{align*}
Therefore,
\begin{equation*}
p_{1,n}(\lambda_n^2)=
\dfrac{ [\kappa \lambda_n^2\sigma_n^\beta +
\alpha\sigma_n-\alpha\sigma_n^2] }{2}
\pm\dfrac{\sqrt{[\kappa \lambda_n^2\sigma_n^\beta +
\alpha\sigma_n-\alpha\sigma_n^2]^2+
4\gamma^2\lambda_n^2\sigma_n^2}}{2}.
\end{equation*}
Choosing positive root and from estimate 
\eqref{OrdemLambdaCase4},  we have

\begin{multline}
\label{OrdemP1Case4}
|p_{1,n}(\lambda_n^2)|
= \Bigg | \dfrac{\sqrt{[\kappa \lambda_n^2
\sigma_n^\beta +\alpha\sigma_n-
\alpha\sigma_n^2]^2+4\gamma^2\lambda_n^2
\sigma_n^2}-\alpha\sigma_n^2}{2}\\+
\dfrac{\kappa\lambda_n^2\sigma_n^\beta+
\alpha\sigma_n}{2} \Bigg |\\
= \Big | \dfrac{\kappa^2\lambda_n^4 \sigma_n^{2\beta}+
\alpha^2\sigma_n^2+2\alpha\kappa\lambda_n^2
\sigma_n^{1+\beta}-2\kappa\alpha\lambda_n^2
\sigma_n^{2+\beta}-2\alpha^2\sigma_n^3+
4\gamma^2\lambda_n^2\sigma_n^2}
{2\sqrt{[\kappa \lambda_n^2\sigma_n^\beta +
\alpha\sigma_n-\alpha\sigma_n^2]^2+
4\gamma^2\lambda_n^2\sigma_n^2}+2\alpha\sigma_n^2}\\
+\dfrac{\kappa\lambda_n^2\sigma_n^\beta+
\alpha\sigma_n}{2} \Big |\approx 
|\sigma_n|^{1+\beta}\approx |\lambda_n|^{2+2\beta}.
\end{multline}
On the other hand, from 
\eqref{OrdemLambdaCase4} and 
\eqref{OrdemP1Case4}, we obtain
$$
|p_{2,n}(\lambda_n^2)|= 
\Big |\dfrac{\gamma^2\lambda_n^2
\sigma_n^2}{p_{1,n}(\lambda_n^2)} 
\Big |\approx |\lambda_n|^{4-2\beta}.
$$
Hence,
  from \eqref{Muexpre01A},  
we have 
\begin{align*}
|\mu_n| & = 
\Big | \dfrac{ (1+\kappa\sigma_n)
\{  p_{2,n}(\lambda_n^2)-
i\delta\lambda_n\sigma_n^\theta\} }
{-i\delta\lambda_n\sigma_n^\theta p_{1,n}
(\lambda^2)} \Big |\\ 
& \approx 
 \left\{ \begin{array}{ccc}
|\lambda_n |^{3-4\beta-2\theta} 
& {\rm for }& 3\geq 2\beta+2\theta\\
|\lambda_n|^{-2\beta} 
& {\rm for } & 3<2\beta+2\theta
\end{array}  \right.
\end{align*}
Finally
\begin{align}
\nonumber
\|U_n\|_\mathcal{H}\geq 
K\|w_n\|_{D(A^{\frac{\beta}{2}})} &  = 
K|\lambda_n||\mu_n|
\underbrace{\|e_n\|_{D(A^\frac{\beta}{2})}}_1\\
\label{Nu007}
& \approx  
\left\{ \begin{array}{ccc}
|\lambda_n |^{4-4\beta-2\theta}
 & {\rm for }& 3\geq 2\beta+2\theta\\
|\lambda_n|^{1-2\beta} 
& {\rm for } & 3<2\beta+2\theta.
\end{array}  \right.
\end{align}
Then  
\begin{equation}\label{Nu008}
|\lambda_n|\|U_n\|_\mathcal{H}\geq 
K \left\{ \begin{array}{ccc}
|\lambda_n |^{5-4\beta-2\theta} 
& {\rm for }& 3\geq 2\beta+2\theta\\
|\lambda_n|^{2-2\beta} 
& {\rm for } & 3<2\beta+2\theta.
\end{array}  \right.
\end{equation} 
As $5-4\beta-2\theta>0\Longleftrightarrow  
\dfrac{5}{2}>2\beta+\theta$  for 
$\dfrac{3}{2}\geq \theta+\beta$ and 
$(\theta,\beta)\in R_{11}$,   
then
\begin{equation}
|\lambda_n|\|U_n\|_\mathcal{H}\to \infty\qquad 
{\rm  in}\qquad R_1.
\end{equation}
Thereupon   
$S(t)$ is not analytic on $R_1$.

 \smallskip
 
 \noindent
{\bf Second Case:} 
  Lack of analyticity  for $\beta=1$,   it follows from \eqref{auto-10A}  that $\|e_n\|_{D(A^\frac{1}{2})}=1$  and $0\leq \theta< \frac{1}{2}$,   using  $\nu_n$.  \\
 
\begin{center}
    \begin{tikzpicture}[scale=2.9, rotate=0]
    \label{Figura03}
    \draw[thick,->] (0,0) -- (1.25,0);
    \draw[thick,->] (0,0) -- (0,1.25);
    \foreach \x in {0,1}
    \draw (\x cm,1pt) -- (\x cm,-1pt) node[anchor=north] {$\x$};
  \draw[thick,->] (0,0) -- (1.25,0)
    node[anchor=north west] {$\theta$ axis};
    \draw[thick,->] (0,0) -- (0,1.25) node[anchor=south east] {$\beta$ axis};
     \draw (0.5,1.1) node[anchor=center] {$(\frac{1}{2}, 1)$};
        \draw (-0.1, 1) node[anchor=center] {$1$};
  \coordinate[label=left:] (a) at (1,0);
    \coordinate[label=right:] (b) at (1,0);
    \coordinate[label=right:] (c) at (1/2,1);
    \coordinate[label=left:] (d) at (1,1);

     \draw[black!5!black, ultra thick] (0,1) -- (1/2, 1);
    \draw[white!5!white, ultra thick] (0.5, 1) -- (1,1);
    \draw[black!25!black, thick] (0, 0) -- (1.2, 0);
    \draw[white!5!white, ultra thick] (1, 0) -- (1, 1);
    \draw[fill=black!100!black, dashed] (0,1) circle (0.025);
   \draw[fill=white!100!white, dashed] (1/2,1) circle (0.025);
    \end{tikzpicture}
    
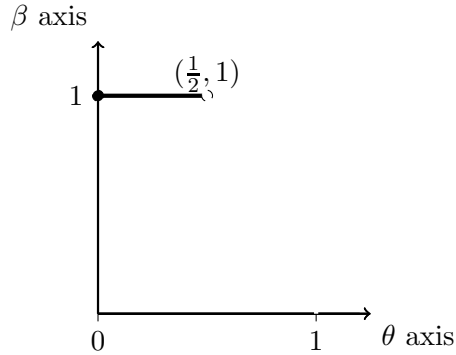
    {\bf Fig.  03}: Region  $R_2:= \{ (\theta,1)\in\R^2/ 
    \theta\in [0,1/2) \}$.
\end{center}
 
   From  algebraic  system  \eqref{eq001systemA}, \eqref{eq002systemA} and \eqref{P1P2A}, we have
\begin{equation}\label{Nuexpre01A}
\nu_n=\frac{\Delta_{\nu_n}}
{p_{1,n}(\lambda^2_ n)p_{2,n}(\lambda^2_n)-\gamma^2\lambda^2_n\sigma_n^2
    -i\delta\lambda_n\sigma_n^\theta  p_{1,n}(\lambda^2_n)},
\end{equation}
where
$$\Delta_{\nu_n}=\left | \begin{array}{cc}
p_{1,n}(\lambda_n^2)&1+\kappa\sigma_n \\
i\gamma\lambda_n\sigma_n & 0
\end{array} \right | =- i\gamma\lambda_n(\sigma_n+\kappa\sigma_n^2).$$
Consequently,  taking again $q_n(s_n)=0$   and    choosing  $s_n^+=\lambda_n^2$, 
we have 
\begin{align}\label{Nuexpre01AB}
|\nu_n|= \Big | \frac{- i\gamma\lambda_n(\sigma_n+\kappa\sigma_n^2)}
{ -i\delta\lambda_n\sigma_n^\theta  p_{1,n}(\lambda^2_n)} \Big |= \Big | \frac{\gamma(\sigma_n+\kappa\sigma_n^2)}
{ \delta\sigma_n^\theta  p_{1,n}(\lambda^2_n)} \Big |\qquad {\rm and}
\end{align}

\begin{equation}
\label{Sn+-AB}
|\lambda_n|^2=|s_n^{+}|=  \Big |\dfrac{-b+\sqrt{b^2-4ac}}{2a} \Big |,\qquad{\rm where}
\end{equation}

\begin{eqnarray}\label{E00lambda-}
  |a|&=&|(1+\kappa\sigma_n)|\approx |\sigma_n|\\
  \label{E01lambda-}
|2c|& =& |2\alpha^2\sigma_n^3|\approx|\sigma_n|^3\\
\label{E02lambda-}
|-b|& = &|\alpha\sigma_n
+(\kappa\alpha+
\alpha+\gamma^2)\sigma_n^2|\approx |\sigma_n|^2\\
\nonumber
|\sqrt{b^2-4ac}|& =& |\{ {\alpha^2\sigma_n^2+2\alpha[\kappa\alpha+\gamma^2-\alpha]\sigma_n^3} \\
\notag
&  &+[(\kappa\alpha-\alpha)^2+\gamma^4+2\gamma^2(\kappa\alpha+\alpha)]\sigma_n^4\}^\frac{1}{2} |\\
\label{E03lambda-}
& & \approx  |\sigma_n|^2.
\end{eqnarray} 
Consequently,  using equations \eqref{E00lambda-}, \eqref{E02lambda-} and \eqref{E03lambda-}
in \eqref{Sn+-AB},  we obtain
\begin{equation}\label{E04lambda-}
|\lambda_n|^2\approx |\sigma_n|.
\end{equation}
And since 
\begin{eqnarray}\nonumber
|p_{1,n}(\lambda_n^2)| &= & \Big | \dfrac{[\alpha\sigma_n+(\kappa\alpha+\alpha+\gamma^2)\sigma_n^2]+\sqrt{b^2-4ac}}{2(1+\kappa\sigma_n)}(1+\kappa\sigma_n)-\alpha\sigma_n^2 \Big |
\\
\label{Case2LackAnaliticity}
&=&
 \Big | \dfrac{\sqrt{b^2-4ac}-\alpha(1-\kappa)\sigma_n^2}{2}+ \dfrac{\alpha\sigma_n+\gamma^2\sigma_n^2}{2}
 \Big |.
\end{eqnarray}
The conclusion could be direct without mentioning the   \eqref{Case2LackAnaliticity}
\begin{equation}\label{Case2ALackAnaliticity}
|p_{1,n}(\lambda_n^2)|\approx|\sigma_n|^2\quad{\rm for}\quad \kappa\geq 1.
\end{equation}
Let's analyze for $\kappa<1$: 
\begin{eqnarray}
\nonumber
|p_{1,n}(\lambda_n^2)| &=&  \Big | \dfrac{b^2-4ac-\alpha^2(1-\kappa)^2\sigma_n^4}{2[\sqrt{b^2-4ac}+\alpha(1-\kappa)\sigma_n^2]}+    \dfrac{\alpha\sigma_n+\gamma^2\sigma_n^2}{2}    \Big |
\\
\nonumber
&=&  \Big |\bigg\{ \dfrac{\alpha^2\sigma_n^2+[2\alpha(\kappa\alpha+\gamma^2-\alpha)]\sigma_n^3} {2[\sqrt{b^2-4ac}+\alpha(1-\kappa)\sigma_n^2]} \\
\nonumber
& & + \dfrac{[\alpha^2(\kappa-1)^2+\gamma^4+2\alpha\gamma^2(1+\kappa)]\sigma_n^4-\alpha^2(1-\kappa)^2\sigma_n^4}
 {2[\sqrt{b^2-4ac}+\alpha(1-\kappa)\sigma_n^2]}\\
 \nonumber
 & & + \dfrac{\alpha\sigma_n+\gamma^2\sigma_n^2}{2}    \Big |
\\
\nonumber
&= &  \Big |\bigg\{ \dfrac{\alpha^2\sigma_n^2+[2\alpha(\kappa\alpha+\gamma^2-\alpha)]\sigma_n^3} {2[\sqrt{b^2-4ac}+\alpha(1-\kappa)\sigma_n^2]} \\
\notag
& & + \dfrac{[\gamma^4+2\alpha\gamma^2(1+\kappa)]\sigma_n^4}
 {2[\sqrt{b^2-4ac}+\alpha(1-\kappa)\sigma_n^2]} +    \dfrac{\alpha\sigma_n+\gamma^2\sigma_n^2}{2}    \Big |
\\
\label{Case2BLackAnaliticity}
& & \approx  |\sigma_n|^2\quad{\rm for}\quad \kappa<1.
\end{eqnarray} 
Then,   for  $\beta=1$ using  \eqref{Case2ALackAnaliticity}  and  \eqref{Case2BLackAnaliticity} in \eqref{Nuexpre01AB},  we obtain
\begin{equation}
\label{Nuexpre02B}
|\nu_n|\approx |\sigma_n|^{-\theta}\approx |\lambda_n|^{-2\theta}\quad{\rm for}\quad \beta=1.
\end{equation}
Therefore
\begin{equation}\label{Nuexpre03B}
\|U_n\|_\mathcal{H}\geq K\|v_n\|_{D(A^\frac{1}{2})}=K|\nu_n|\underbrace{\|e_n\|_{D(A^\frac{1}{2})}}_1=K|\lambda_n|^{-2\theta}\quad{\rm for}\quad \beta=1.
\end{equation}
 Finally,  
\begin{equation}\label{LackCase02}
|\lambda_n|\|U_n\|_\mathcal{H}\geq K|\lambda_n|^{1-2\theta} \quad {\rm for }\quad \beta=1.
\end{equation}

As  $1-2\theta>0$ for  $0\leq\theta<\frac{1}{2}$,  then  $\lim\limits_{|\lambda|\to\infty}|\lambda_n|\|U_n\|_\mathcal{H}\to \infty$. \\
Then $S(t)$  is non-analytic  in $R_2$.

 \smallskip
 
 \noindent
{\bf  Third case:} 
 
\begin{center}
    \begin{tikzpicture}[scale=3, rotate=0]
    \label{Figura04}
    \draw[thick,->] (0,0) -- (1.25,0);
    \draw[thick,->] (0,0) -- (0,1.25);
    \foreach \x in {0,1}
    \draw (\x cm,1pt) -- (\x cm,-1pt) node[anchor=north] {$\x$};
    \foreach \y in {0,1}
 \draw[thick,->] (0,0) -- (1.25,0)
    node[anchor=north west] {$\theta$ axis};
    \draw[thick,->] (0,0) -- (0,1.25) node[anchor=south east] {$\beta$ axis};
       \draw (-0.1, 1) node[anchor=center] {$1$};
     \draw(0.5,1.1) node[anchor=center] {$(\frac{1}{2}, 1)$};
 \coordinate[label=left:] (a) at (1,0);
    \coordinate[label=right:] (b) at (1,0);
    \coordinate[label=right:] (c) at (1/2,1);
    \coordinate[label=left:] (d) at (1,1);
 \draw[black!5!black, thick, dashed] (0,1) -- (1/2, 1);
    \draw[black!95!black, ultra thick ] (1/2,1) -- (1, 1);
     \draw[black!25!black, thick, dashed] (1,0) -- (1, 1);
  \draw[black!25!black, thick] (0, 0) -- (1.2, 0);
  \draw[fill=black!100!black, dashed] (1,1) circle (0.025);
   \draw[fill=white!100!white, dashed] (1/2,1) circle (0.025);
   \end{tikzpicture}
    
    {\bf Fig.  04}: Region $R_3:=\{ (\theta,1)\in\R^2/ 
    \theta\in (\frac{1}{2},1] \}$.
\end{center}

 For  $\beta=1$  and $\frac{1}{2}<\theta\leq 1$,   let's assuming  that $p_{1,n}(s_n)=0$ and determine the order of $|\mu_n|$:
$$
p_{1,n}(s_n)=0\Longleftrightarrow \lambda_n^2=\dfrac{\alpha\sigma_n^2}{1+\kappa\sigma_n}\Longrightarrow  |\lambda_n|^2\approx|\sigma_n|.$$
As  
\begin{equation*}
|p_{2,n}(s_n)| =|\lambda_n^2-\alpha\sigma_n| = \Big | \dfrac{\alpha(1-\kappa)\sigma_n^2-\alpha\sigma_n}{1+\kappa\sigma_n} \Big |
\approx \left\{\begin{array}{ccc}
|\sigma_n|^0&{\rm if} & \kappa=1\\
|\sigma_n|^1  &{\rm if} & \kappa\not= 1. 
\end{array}\right.
\end{equation*}

From $\frac{1}{2}<\theta\leq 1$, we have $4<3+2\theta$,   then
\begin{equation}\label{Muexpre01C}
|\mu_n|= \Big | \frac{(1+\kappa\sigma_n)  \big\{p_{2,n}(\lambda^2_n)
    -i\delta\lambda_n\sigma_n^\theta\big\}}
{-\gamma^2\lambda^2_n\sigma_n^2} \Big |\approx|\lambda_n|^{2\theta-3}.
\end{equation}
Finally, 
\begin{equation}\label{Nuexpre03C}
\|U_n\|_\mathcal{H}\geq K\|w_n\|_{D(A^\frac{1}{2})}=K|\lambda_n||\mu_n|\underbrace{\|e_n\|_{D(A^\frac{1}{2})}}_1=K|\lambda_n|^{2\theta-2}.
\end{equation}
Then 
\begin{equation}\label{LackCase03}
|\lambda_n|\|U_n\|_\mathcal{H}\geq K|\lambda_n|^{2\theta-1} \quad {\rm for }\quad \beta=1.
\end{equation}
As  $2\theta-1>0 \Longleftrightarrow \frac{1}{2}<\theta\leq 1$. Then $|\lambda_n|\|U_n\|_\mathcal{H}\to \infty$ for $ \frac{1}{2}<\theta\leq 1$ and $\beta=1$.  
Therefore,  $S(t)$ is not analytical on $R_3$.

 \smallskip
 
 \noindent
{\bf Fourth Case:}
Considering initially the region 
$R_{11}=\{ (\theta,\beta)\in \R^2\; /\; 0\leq 
\theta\leq 1\quad{\rm and}\quad 0<\beta<1  \}$ 
  
\begin{center}
    \begin{tikzpicture}[scale=3, rotate=0]
    \label{Figura05}
    \draw[thick,->] (0,0) -- (1.25,0);
    \draw[thick,->] (0,0) -- (0,1.25);
    \foreach \x in {0,1}
    \draw (\x cm,1pt) -- (\x cm,-1pt) node[anchor=north] {$\x$};
    \foreach \y in {0,1}
  \draw[thick,->] (0,0) -- (1.25,0)
    node[anchor=north west] {$\theta$ axis};
    \draw[thick,->] (0,0) -- (0,1.25) node[anchor=south east] {$\beta$ axis};
           \draw (-0.1, 1) node[anchor=center] {$1$};
     \draw(0.5,1.11) node[anchor=center] {$(\frac{1}{2}, 1)$};
 \draw[black!95!black, ultra thick] (0,0) -- (1, 0);
    \coordinate[label=left:] (a) at (1/2,1);
    \coordinate[label=right:] (b) at (1,1);
    \coordinate[label=right:] (c) at (1,0);
  \draw[rectangle, fill=black!20!white] (a) -- (b) -- (c) -- (a) -- cycle;
\draw[black!5!white, ultra thick] (0,0) -- (1, 0);
  \draw[black!5!black, ultra thick] (1,0) -- (1, 1);
  \draw[white!5!white, ultra thick, dashed] (1/2,1) -- (1, 1);
   \draw[white!5!white, ultra thick, dashed] (1/2,1) -- (1, 0);
 \draw[black!25!black, thick] (0,0) -- (1, 0);
   \draw[black!25!black, thick] (0, 0) -- (0, 1);
    \draw[white!05!white, ultra thick] (0, 1) -- (1/2, 1);
     \draw[black!5!black, thick, dashed] (0,1) -- (1/2, 1);
  \draw[fill=white!100!white, dashed] (1/2,1) circle (0.025);
   \draw[fill=white!100!white, dashed] (1,1) circle (0.025);
    \draw[fill=white!100!white dashed] (1,0) circle (0.025);
    \end{tikzpicture}

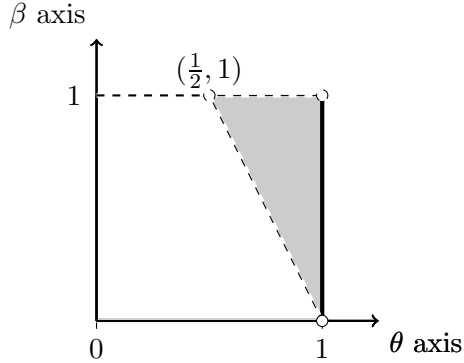
    {\bf Fig.  05}: Region $R_4:=\{(\theta,\beta)\in (\frac{1}{2},1]\times(0,1)/ 2<2\theta+\beta\}$.
\end{center}

For $(\theta,\beta)\in R_4$.  
Let us assume that
$p_{1,n}(\lambda_n)=0$, then
\begin{multline}\label{Ordem5to001}
|\lambda_n|^2  = \Big |\dfrac{\alpha\sigma_n^2}{1+\kappa\sigma_n^\beta} \Big |\Longrightarrow |\lambda_n|^2\approx |\sigma_n|^{2-\beta}\\
{\rm and}\quad |p_{2,n}(\lambda_n^2)|=|\lambda_n^2-\alpha\sigma_n|\approx|\sigma_n|^{2-\beta}\approx|\lambda_n|^2.
\end{multline}
Using \eqref{Ordem5to001} in \eqref{Muexpre01A}, we have
\begin{multline}
\label{Ordem5to002}
|\mu_n|= \Big |\dfrac{(1+\kappa\sigma_n^\beta)\{p_{2,n}(\lambda_n^2)-i\delta\lambda_n\sigma_n^\theta\}}{-\gamma^2\lambda_n^2\sigma_n^2} \Big |\\
\approx\left\{\begin{array}{ccc}
|\lambda_n|^\frac{2\beta-4}{2-\beta}  &{\rm for} & 2>\beta+2\theta \\
|\lambda_n|^\frac{3\beta+2\theta-6}{2-\beta}  & {\rm for} & 2\leq\beta+2\theta
\end{array}\right.
\end{multline}
Finally, 
\begin{multline}
\label{Nu009}
\|U_n\|_\mathcal{H}\geq 
K\|w_n\|_{D(A^{\frac{\beta}{2}})}  = 
K|\lambda_n||\mu_n|
\underbrace{\|e_n\|_{D(A^\frac{\beta}{2})}}_1
\\
\approx   
\left\{ \begin{array}{ccc}
|\lambda_n |^{-1} 
& {\rm for }& 2\geq \beta+2\theta\\
|\lambda_n|^\frac{2\beta+2\theta-4}{2-\beta} 
& {\rm for } & 2<\beta+2\theta
\end{array}  \right.
\end{multline}
Then 
\begin{equation}\label{Nu010}
|\lambda_n|\|U_n\|_\mathcal{H}\geq K 
\left\{ \begin{array}{ccc}
|\lambda_n |^0 
& {\rm for }& 2\geq  \beta+2\theta\\
|\lambda_n|^\frac{\beta+2\theta-2}{2-\beta} 
& {\rm for } & 2<\beta+2\theta
\end{array}  \right.
\end{equation}
As $\beta+2\theta-2>0\Longleftrightarrow 
2<\beta+2\theta$,  
then 
$|\lambda_n|\|U_n\|_\mathcal{H}\to\infty$ 
for 
$2<\beta+2\theta$ in $R_{11}$. 
Therefore,   
$S(t)$ is not analytic on $R_4$.
\end{proof}

\subsection{Gevrey Class of $S(t)=e^{\mathbb{B}t}$ in $R_G$}
\label{3.2}
In this subsection we determine the Gevrey  class of the semigroup $S(t)=e^{\mathbb{B}t}$ in the region $R_G:=\{(\theta,\beta)\in (0, 1] \times (0,1)\}$, and to  this purpose,  we need to recall the definition and result presented in \cite{SCRT1990,Tebou-2020} (adapted from \cite{TaylorM}, Theorem 4, p. 153).

\begin{definition}\label{Def1.1Tebou} Let $t_0\geq 0$ be a real number.  A strongly continuous semigroup $S(t)$, defined on a Banach space $ \mathcal{H}$, is of Gevrey class $s > 1$ for $t > t_0$, if $S(t)$ is infinitely differentiable for $t > t_0$, and for every compact set $K \subset (t_0,\infty)$ and for each $\mu > 0$,  there exists a constant $ C = C(\mu, K) > 0$ such that
    \begin{equation}\label{DesigDef1.1}
    ||S^{(n)}(t)||_{\mathcal{L}( \mathcal{H})} \leq  C\mu ^n(n!)^s,  \text{ for all } \quad t \in K, n = 0,1,2...
    \end{equation}
\end{definition}
 
\begin{theorem}[\cite{TaylorM}]\label{Theorem1.2Tebon}
    Let $S(t)$  be a strongly continuous and bounded semigroup on a Hilbert space $ \mathcal{H}$.  Suppose that the infinitesimal generator $\mathbb{B}$ of the semigroup $S(t)$ satisfies the following estimate, for some $0 < \phi < 1$:
    \begin{equation}\label{Eq1.5Tebon2020}
    \lim\limits_{|\lambda|\to\infty} \sup |\lambda |^\phi \|(i\lambda I-\mathbb{B})^{-1}\|_{\mathcal{L}( \mathcal{H})} < \infty.
    \end{equation}
    Then $S(t)$  is of Gevrey  class  $s$   for $t>0$,  for every   $s >\dfrac{1}{\phi}$.
\end{theorem}
 
It was necessary to subdivide our study of the Gevrey class of $S(t)$ into two subsubsections. In subsubsection \eqref{3.2.1} we determine the Gevrey class $s_1> \frac{1}{\Phi_{s_1}}$ for the region $ R_{G1}:= \{(\theta,\beta)\in\mathbb{R}^2/ 2\leq \theta+2\beta, \; 0< \theta\leq 1\;{\rm and}\; 0<\beta<1\}$ and in subsubsection \eqref{3.2.2} we determine the Gevrey class $s_2> \frac{1}{\phi_2}=\frac{2(2+\theta-\beta)}{\theta}$ for the region $R_{G2} :=\{(\theta,\beta)\in\mathbb{R}^2/ \frac{3\theta}{4}+\frac{\beta}{2}\leq 1,\; 0<\theta\leq 1\;{\rm and}\;0<\beta<1 \}$. We highlight that the Gevrey Region of the semigroup $S(t)$, is given by: $R_G = R_{G1}\cup R_{G2}$. Since $R_{G1}\cap R_{G2}$ is not empty, in this intersection region the Gevrey class $s_i$  will be $s_i> \frac{1}{\max\{\Phi_{s_1},\phi_2\}}$.
 
\subsubsection{Gevrey class $s_1 > \frac{1}{\Phi_{s_1}}=\frac{1}{2\max\{ \frac{1-\beta}{3-\beta}, \frac{\theta}{2+\theta-\beta}\}}$  for  $2\leq \theta+2\beta$.}
\label{3.2.1}
 
\begin{center}
    \begin{tikzpicture}[scale=3, rotate=0]
    \label{Figura06}
    \draw[thick,->] (0,0) -- (1.25,0);
    \draw[thick,->] (0,0) -- (0,1.25);
    \foreach \x in {0,1}
    \draw (\x cm,1pt) -- (\x cm,-1pt) node[anchor=north] {$\x$};
    \foreach \y in {0,1}
   \draw[thick,->] (0,0) -- (1.25,0)
    node[anchor=north west] {$\theta$ axis};
    \draw[thick,->] (0,0) -- (0,1.25) node[anchor=south east] {$\beta$ axis};
           \draw (-0.1, 1) node[anchor=center] {$1$};
    \coordinate[label=left:] (a) at (0,1);
    \coordinate[label=right:] (b) at (1,1);
    \coordinate[label=left:] (c) at (1,1/2);

    \draw[rectangle, fill=black!30!white] (a)  -- (b) --  (c) -- (a) -- cycle;
   \draw(1.08,0.54) node[anchor=center] {$\qquad(1, \frac{1}{2})$};
  \draw[black!10!white, ultra thick] (0, 1) -- (1, 1/2); 
     \draw[blue!10!white, ultra thick, dashed] (0, 1) -- (1, 1);
    \draw[white!5!white, thick] (1, 0) -- (1, 1/2);
    \draw[black!25!black, ultra thick] (1, 1/2) -- (1, 1);
    \draw[black!25!black, ultra thick] (0, 1) -- (1, 1/2);
     \draw[black!25!black, ultra thick] (0, 1) -- (1, 1/2); \draw[black!5!black, thick, dashed] (1,0) -- (1,1/2);
    \draw[fill=black!95!black, dashed] (1,1/2) circle (0.02);
    \draw[fill=white!95!black, dashed] (1,1) circle (0.02);
  \draw[fill=white!100!white, dashed] (0,1) circle (0.025);
    \end{tikzpicture}

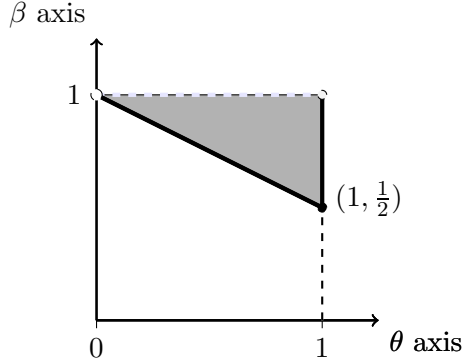
    {\bf Fig.  06}: Gevrey region 
    $R_{G1}:= \{(\theta,\beta)\in\mathbb{R}^2/ 
    2\leq \theta+2\beta, \; 0< \theta\leq 1\;
    {\rm and}\; 0<\beta<1\}$.
\end{center}
 
First, we are going to show a lemma to obtain 
an estimation necessary to control the term 
$|\lambda|\|A^\frac{\beta}{2}w\|^2$.
 
\begin{lemma}\label{Lemma01Gevrey}
Let  $\tau> 0$.  There exists a positive constant $C_\tau$ such that the solutions of system \eqref{ISplacas-20}--\eqref{ISplacas-40}, for $|\lambda| \geq \tau$, satisfy the following inequalities:
\begin{equation}\label{Eq07EstAbeta2G}
\|A^\frac{1}{2}w\|^2   \leq  C_\tau \|F\|_\mathcal{H}\|U\|_\mathcal{H}\qquad{\rm for}\qquad 2\leq \theta+2\beta.
\end{equation}
\end{lemma}

\begin{proof}[\bf Proof.]
Now, applying the duality product between \eqref{esp-40} and    $(I+\kappa A^\beta) A^{-\beta}w$ and in addition  use,  \eqref{esp-10}  and   \eqref{esp-20} and we use the self-adjointness of the powers of the operator $A$ in order to rewrite some terms. Then, we have
\begin{align*}
\gamma\kappa\|A^\frac{1}{2}w\|^2  = &-\dual{A^{-\beta}z}{i\lambda(I+\kappa A^\beta)w}+\alpha\dual{A^\frac{\theta}{2}v}{A^{1-\beta-\frac{\theta}{2}}(i\lambda u-f_1)}
\\
 + &\alpha\kappa\dual{A^\frac{\theta}{2} v}{A^\frac{2-\theta}{2}(i\lambda u-f_1)}+\delta\dual{A^\frac{\theta}{2}z}{(I+\kappa A^\beta)A^{\frac{\theta}{2}-\beta}w}
 \\
- &\gamma\|A^\frac{1-\beta}{2}w\|^2 -\dual{f_4}{(I+\kappa A^\beta)A^{-\beta}w},
\end{align*}
then
\begin{align*}
\gamma\kappa\|A^\frac{1}{2}w\|^2 = &-\dual{A^{-\beta}z}{i\lambda(I+\kappa A^\beta)w}-\alpha\dual{A^\frac{\theta}{2}(i\lambda v)}{A^{1-\beta-\frac{\theta}{2}}u}
\\
+ & \delta\kappa\dual{A^\frac{\theta}{2}z}{A^\frac{\theta}{2}w}   +\delta\dual{A^\frac{\theta}{2}z}{A^{\frac{\theta}{2}-\beta}w}- \alpha\dual{A^\frac{\theta}{2} v}{A^{1-\beta-\frac{\theta}{2}} f_1}\\
-& \alpha\kappa\dual{A^\frac{\theta}{2}(i\lambda v)}{A^\frac{2-\theta}{2} u}-\alpha\kappa\dual{A^\frac{\theta}{2}v}{A^\frac{2-\theta}{2}f_1}\\
-&\gamma\|A^\frac{1-\beta}{2}w\|^2-\dual{f_4}{(I+\kappa A^\beta)A^{-\beta}w}.
\end{align*}

Therefore,  taking real part,   we get
\begin{align}
\notag
\|A^\frac{1}{2}w\|^2  &\leq C{\rm Re}\{-\dual{A^{-\beta}z}{i\lambda(I+\kappa A^\beta)w}-\alpha\dual{A^\frac{\theta}{2}(i\lambda v)}{A^{1-\beta-\frac{\theta}{2}}u}\\
\notag
- & \alpha\dual{A^\frac{\theta}{2} v}{A^{1-\beta-\frac{\theta}{2}} f_1}
-\alpha\kappa\dual{A^\frac{\theta}{2}(i\lambda v)}{A^\frac{2-\theta}{2} u}-\alpha\kappa\dual{A^\frac{\theta}{2}v}{A^\frac{2-\theta}{2}f_1}\\ 
\label{Eq01EstAbeta2G}
+ & \delta\dual{A^\frac{\theta}{2}z}{A^{\frac{\theta}{2}-\beta}w}
+\delta\kappa\dual{A^\frac{\theta}{2}z}{A^\frac{\theta}{2}w}-\dual{f_4}{(I+\kappa A^\beta)A^{-\beta}w}\}.
\end{align}
From \eqref{esp-30}, we have $i\lambda(I+\kappa A^\beta)w=-\alpha A^2u-\gamma Az+(I+\kappa A^\beta)f_3$,  and estimate \eqref{dis-10},    from $\varepsilon>0$, exist $C_\varepsilon>0$ such that
\begin{multline}
\label{Eq02EstAbeta2G}
|\dual{A^{-\beta}z}{i\lambda(I+\kappa A^\beta)w}|= |\dual{A^{-\beta}z}{-\alpha A^2u-\gamma Az+(I+\kappa A^\beta)f_3}|\\
 =  |-\alpha\dual{A^\frac{\theta}{2}z}{A^{2-\frac{\theta}{2}-\beta}u}-\gamma\|A^\frac{1-\beta}{2}z\|^2+\dual{z}{(I+\kappa A^\beta)A^{-\beta}f_3}|\\
\leq   C_\varepsilon\|A^\frac{\theta}{2}z\|^2+\varepsilon\|A^{2-\frac{\theta}{2}-\beta}u\|^2+C\|A^\frac{1-\beta}{2}z\|^2 +C\|U\|_\mathcal{H}\|F\|_\mathcal{H}
\end{multline}
Now,  using estimative \eqref{Eq02EstAbeta2G} in \eqref{Eq01EstAbeta2G},  taking account the  estimate \eqref{dis-10},  applying Cauchy-Schwarz  and Young inequalities, we have
\begin{align}
\notag
\|A^\frac{1}{2}w\|^2  \leq &  C_\varepsilon\|A^\frac{\theta}{2}z\|^2+\varepsilon\|A^{2-\frac{\theta}{2}-\beta}u\|^2+C\|A^\frac{1-\beta}{2}z\|^2+C\|U\|_\mathcal{H}\|F\|_\mathcal{H}\\
\notag
+& C\{ |\dual{A^\frac{\theta}{2} z}{A^{1-\beta-\frac{\theta}{2}}u}| +|\dual{A^\frac{\theta}{2}z}{A^\frac{2-\theta}{2}u}|+|\dual{A^\frac{1}{2}f_2}{A^\frac{1}{2}u}|\\
\label{Eq03EstAbeta2G}
+& |\dual{A^\frac{1}{2}f_2}{A^{-\beta+\frac{1}{2}}u}| 
+|\dual{A^\frac{1}{2}v}{A^{\frac{1}{2}-\beta}
f_1}|+|\dual{A^\frac{1}{2}v}{A^\frac{1}{2}f_1}|\\
\notag
+& |\dual{A^\frac{\theta}{2}z}{A^\frac{\theta}{2}w}|+|\dual{A^\frac{\theta}{2}z}{A^{\frac{\theta}{2}-\beta}w}|
 +|\dual{f_4}{A^{-\beta}w}|+|\dual{f_4}{w}|\}.
\end{align}
 From $2\leq \theta+2\beta\Longleftrightarrow 2-\frac{\theta}{2}-\beta\leq 1$,  $1\leq \theta+\beta\Longleftrightarrow \frac{1-\beta}{2}\leq\frac{\theta}{2}$ and  as $1-\beta-\frac{\theta}{2}\leq 1$,   $ \frac{2-\theta}{2}\leq 1$,   $\frac{\theta}{2}\leq \frac{1}{2}$ and $ \frac{\theta}{2}-\beta\leq\frac{1}{2}$  into account  of the continuous embedding   $D(A^{\theta_2}) \hookrightarrow D(A^{\theta_1}),\;\theta_2>\theta_1$,     for  $\varepsilon>0$, there exist $C_\varepsilon>0$  such that
\begin{equation}
\label{Eq06EstAbeta2G}
\|A^\frac{1}{2}w\|^2   \leq \varepsilon\|Au\|^2+C\|F\|_\mathcal{H}\|U\|_\mathcal{H}\quad{\rm for}\quad 2\leq\theta+2\beta\;{\rm and}\; 1\leq \theta+\beta.
\end{equation}
Note that as the intersection of the regions $2\leq \theta+2\beta$ with $1\leq \theta+\beta$ in our region of interest $R_G=\{ (\theta,\beta)\in (0,1] \times(0,1) \}$ results in the region $2\leq \theta+2\beta$, so the estimate \eqref{Eq06EstAbeta2G} holds for $2\leq \theta+2\beta$.

From estimate  \eqref{Eq05EstAbeta2},   we   finish the proof of the lemma.
\end{proof}
Our main result in this subsubsection \eqref{3.2.1} is as follows:
\begin{theorem} Let  $S(t)=e^{\mathbb{B}t}$  strongly continuos semigroups of contractions on the Hilbert space $ \mathcal{H}$, the semigroups $S(t)$ is of Gevrey class $s$ for every $s>\dfrac{1}{\Phi_{s_1}}$ for $\Phi_{s_1}\in(0,1)$, as there exists a positive constant $C$ such that we have the resolvent estimate:
    \begin{equation}\label{Eq1.6Tebon2020}
   \lim\limits_{|\lambda|\to\infty} \sup  |\lambda |^{\Phi_{s_1}} \|(i\lambda I-\mathbb{B})^{-1}\|_{\mathcal{L}( \mathcal{H})} < \infty.
    \end{equation}
\end{theorem}

\begin{proof}[\bf Proof.]
To show \eqref{Eq1.6Tebon2020}, it is sufficient  to prove that 
\begin{equation}\label{Eq1.6Tebon2020A}
|\lambda|^{\Phi_{s_1}}\|U\|_\mathcal{H}\leq C_\tau\|F\|_\mathcal{H}
\end{equation}
for $\Phi_{s_1}\in (0,1)$.   And note that \eqref{Eq1.6Tebon2020A} implies
\begin{equation}\label{Eq1.6Tebon2020EA}
|\lambda|^{\Phi_{s_1}}\|(i\lambda I-\mathbb{B})^{-1}\|_{\mathcal{L}(\mathcal{H})} \leq |\lambda|^{\Phi_{s_1}}\dfrac{\|U\|_\mathcal{H}}{\|F\|_\mathcal{H}}\leq C_\tau.
\end{equation}
 
 \smallskip
 
 \noindent
 {\bf  Estimating term $|\lambda|\|z\|$.}
   Let us assume $\lambda\in \mathbb{R}$ with  $|\lambda|>1$, we shall borrow some ideas from \cite{LiuR95}.  Let us decompose $z$ as  $z=z_1+z_2$,  where $z_1\in D(A)$ and $z_2\in D(A^0)$, with
    \begin{equation}\label{Eq110AnalyRR}
    i\lambda z_1+ Az_1=f_4\quad{\rm and}\quad  i\lambda z_2=-\alpha Av+\gamma Aw-\delta A^\theta z+Az_1.
    \end{equation}
    Firstly,  applying the duality product  in  the first equation of \eqref{Eq110AnalyRR} by $z_1$, then by $A z_1$, and recalling that the operator $A$ is seft-adjoint,  we have
   \begin{equation}\label{Eq112AnalyRR}
    |\lambda|\|z_1\|+|\lambda|^\frac{1}{2}\|A^\frac{1}{2}z_1\|+\|Az_1\|\leq C\|F\|_\mathcal{H}.
    \end{equation}
 Applying the operator $A^{\frac{\beta}{2}-1}$ in the second equation of \eqref{Eq110AnalyRR}, we have 
    \begin{equation*}
    i\lambda A^{\frac{\beta}{2}-1}z_2=-\alpha A^\frac{\beta}{2} v+\gamma A^\frac{\beta}{2}w-\delta A^{\theta+\frac{\beta}{2}-1} z+ A^\frac{\beta}{2}z_1
    \end{equation*}
    then, as $\theta+\frac{\beta}{2}-1\leq \frac{\theta}{2}$,  $\frac{\beta}{2}\leq\frac{1}{2}$, applying  continuous embedding,  we get
    \begin{equation*} 
    |\lambda|^2\|A^{\frac{\beta}{2}-1} z_2\|^2\leq C\{\|A^\frac{1}{2}v\|^2+\|A^\frac{\beta}{2}w\|+\|A^\frac{\theta}{2}z\|^2\}+C\|A^\frac{1}{2}z_1\|^2,
    \end{equation*}
    from \eqref{ExponentialP1},  inequality \eqref{Eq112AnalyRR} and from the estimates \eqref{dis-10} and \eqref{dis-10A}, for $0\leq\theta\leq 1$ and $0<\beta\leq 1$ and as $-1\leq-\frac{2\theta}{2+\theta-\beta}$,  we obtain
    \begin{align*}
        |\lambda|^2\|A^{\frac{\beta}{2}-1}z_2\|^2 \leq & C\{\|F\|_\mathcal{H}\|U\|_\mathcal{H}+|\lambda|^{-1}\|F\|^2_\mathcal{H}\}\\
        \leq & C|\lambda|^{-\frac{2\theta}{2+\theta-\beta}}\{|\lambda|^\frac{2\theta}{2+\theta-\beta}\|F\|_\mathcal{H}\|U\|_\mathcal{H}+\|F\|^2_\mathcal{H}\}
    \end{align*}
    then, 
    \begin{equation}\label{Eq113AnalyRR}
    \|A^{\frac{\beta}{2}-1}z_2\|^2\leq C|\lambda|^\frac{2\beta-4\theta-4}{2+\theta-\beta}\{|\lambda|^\frac{2\theta}{2+\theta-\beta}\|F\|_\mathcal{H}\|U\|_\mathcal{H}+\|F\|^2_\mathcal{H}\}.
    \end{equation}
    On the  other hand, from $z_2=z-z_1$,  \eqref{dis-10},  as $\frac{\theta}{2}\leq\frac{1}{2}$ and using continuous embedding in inequality  \eqref{Eq112AnalyRR},  we  have
    \begin{align} \notag
    \|A^\frac{\theta}{2} z_2\|^2   \leq & C\{\|A^\frac{\theta}{2} z\|^2+\|A^\frac{\theta}{2}z_1\|^2\}\\
    \label{Eq114AnalyRR}
    \leq &  C|\lambda|^{-\frac{2\theta}{2+\theta-\beta}}\{|\lambda|^\frac{2\beta}{2+\theta-\beta}\|F\|_\mathcal{H}\|U\|_\mathcal{H}+\|F\|^2_\mathcal{H}\}.
    \end{align}
    Now, by Lions' interpolations inequality($\frac{\beta}{2}-1< 0\leq\frac{\theta}{2}$),  we derive
    \begin{equation}\label{Eq115AnalyRR}
     \|z_2\|^2\leq C(\|A^{\frac{\beta}{2}-1}z_2\|^2)^\frac{\theta}{2+\theta-\beta}(\|A^\frac{\theta}{2}z_2\|^2)^\frac{2-\beta}{2+\theta-\beta}.
    \end{equation}
    From \eqref{Eq113AnalyRR}, we have
    \begin{multline}\label{Eq116AnalyRR}
    (\|A^{\frac{\beta}{2}-1}z_2\|^2)^\frac{\theta}{2+\theta-\beta} \leq  C|\lambda|^{-\big(\frac{4+4\theta-2\beta}{2+\theta-\beta}\big)\big(\frac{\theta}{2+\theta-\beta}\big)}\{|\lambda|^\frac{2\theta}{2+\theta-\beta}\|F\|_\mathcal{H}\|U\|_\mathcal{H}\\
    +\|F\|^2_\mathcal{H}\}^\frac{\theta}{2+\theta-\beta},
    \end{multline}
  and from \eqref{Eq114AnalyRR},  one can have
    \begin{multline}\label{Eq117AnalyRR}
    (\|A^\frac{\theta}{2} z_2\|^2)^\frac{2-\beta}{2+\theta-\beta}\leq C|\lambda|^{-\big(\frac{2\theta}{2+\theta-\beta}\big)\big(\frac{2-\beta}{2+\theta-\beta}\big)}\{|\lambda|^\frac{2\theta}{2+\theta-\beta}\|F\|_\mathcal{H}\|U\|_\mathcal{H}\\
    +\|F\|^2_\mathcal{H}\}^\frac{2-\beta}{2+\theta-\beta}.
    \end{multline}
    Then, using \eqref{Eq116AnalyRR} and \eqref{Eq117AnalyRR} in \eqref{Eq115AnalyRR}, we derive
    \begin{equation}\label{Eq118AnalyRR}
     \|z_2\|^2\leq C|\lambda|^{-\frac{4\theta}{2+\theta-\beta}}\{|\lambda|^\frac{2\theta}{2+\theta-\beta} \|F\|_\mathcal{H}\|U\|_\mathcal{H}+\|F\|^2_\mathcal{H}\}.
    \end{equation}
    Therefore as $\|z\|^2\leq C\{\|z_1\|^2+\|z_2\|^2\}$  and $-2\leq -\frac{4\theta}{2+\theta-\beta}$, from  inequality  \eqref{Eq112AnalyRR} and \eqref{Eq118AnalyRR}, we have
    \begin{equation*}
    |\lambda|\|z\|^2\leq C|\lambda|^\frac{2-3\theta-\beta}{2+\theta-\beta}\{|\lambda|^\frac{2\theta}{2+\theta-\beta}\|F\|_\mathcal{H}\|U\|_\mathcal{H}+\|F\|^2_\mathcal{H}\}.
    \end{equation*}
Therefore  
    \begin{equation}\label{Eq119AnalyRR}
    |\lambda|^\frac{2\theta}{2+\theta-\beta}\|z\|\leq C\|F\|_\mathcal{H},\quad {\rm for}\quad  0\leq \theta\leq 1\quad{\rm and}\quad 0<\beta\leq 1.
    \end{equation}
    Conversely,   we will now estimate the term 
$|\lambda|\|A^\frac{\beta}{2}w\|^2$,  
assuming that $\lambda\in\R$ with  
$|\lambda|>1$.  Let us decompose 
$w$ as  
$w=w_1+w_2$, where 
$w_1\in D(A)$ and 
$w_2\in D(A^\frac{\beta}{2})$, with
    \begin{multline}\label{Eq110AnalyR}
    i\lambda(I+\kappa A^\beta )w_1+Aw_1=(I+\kappa A^\beta)f_3,\\ i\lambda(I+\kappa A^\beta )w_2=-\alpha A^2u-\gamma Az+Aw_1.
    \end{multline}
    Firstly,  applying the product duality of  the first equation in \eqref{Eq110AnalyR} by $w_1$, we have
    \begin{equation}\label{Eq111AnalyR}
    i\lambda\|w_1\|^2+i\lambda\kappa\|A^\frac{\beta}{2}w_1\|^2+ \|A^\frac{1}{2} w_1\|^2=\dual{f_3}{w_1}+\kappa\dual{A^\frac{\beta}{2}f_3} {A^\frac{\beta}{2}w_1}.
    \end{equation}
 Taking the imaginary part of  \eqref{Eq111AnalyR} first  and subsequently the real part and applying Cauchy-Schwarz inequality,  we have
    \begin{multline*}
    |\lambda|\|w_1\|^2+\kappa|\lambda|\|A^\frac{\beta}{2}w_1\|\leq |\rm{Im}\dual{f_3}{w_1}|+\kappa|\rm{Im}\dual{A^\frac{\beta}{2}f_3}{A^\frac{\beta}{2}w_1}|\\
   {\rm and}\qquad \|A^\frac{1}{2}w_1\|\leq C\|F\|_\mathcal{H}.
    \end{multline*}
    Equivalently, 
    \begin{equation}\label{Eq112AnalyR}
  |\lambda| \|A^\frac{\beta}{2}w_1\|\leq C\|F\|_\mathcal{H}\qquad{\rm and}\qquad \|A^\frac{1}{2}w_1\|\leq C\|F\|_\mathcal{H}.
    \end{equation}
 From $A^\frac{\beta}{2}w_2=A^\frac{\beta}{2}w-A^\frac{\beta}{2}w_1$, using \eqref{Eq112AnalyR}$_2$, we have
\begin{eqnarray*}
\|A^\frac{\beta}{2}w_2\|^2 &\leq & C \{\|A^\frac{\beta}{2}w\|^2+\|A^\frac{\beta}{2}w_1\|^2\}.
\end{eqnarray*}
From  \eqref{Eq06EstAbeta2} and second inequality  of  \eqref{Eq112AnalyR} for  $\tau>0$ there exists a positive constant $C_\tau$ such that for $|\lambda|\geq\tau$ we obtain
    \begin{equation}\label{Eq01EstW}
    \|A^\frac{\beta}{2}w_2\|^2\leq C_\tau\{ \|F\|_\mathcal{H}\|U\|_\mathcal{H}+|\lambda|^{-2}\|F\|^2\}
    \end{equation}
 From second equation of  \eqref{Eq110AnalyR}, we have
    \begin{eqnarray*}
|\lambda|\|(I+\kappa A^\beta)[A^{-1}w_2]|| & \leq &C\{ \|Au\|+\|z\|\}+\|w_1\|,
    \end{eqnarray*}
    then, we  find
    \begin{equation*}
    |\lambda|(\|A^{-1} w_2\|^2+\kappa\|A^{\frac{\beta}{2}-1}w_2\|^2)^\frac{1}{2} \leq C\{\|Au\|+\|z\|\}+\|w_1\|,
    \end{equation*}
    applying Cauchy-Schwarz and Young inequalities and using first inequality of \eqref{Eq112AnalyR} and estimates \eqref{dis-10},  \eqref{Eq05EstAbeta2}  for $|\lambda|^\frac{1}{2}>1$,  for $0\leq\theta\leq 1$ and $0<\beta\leq 1$,  we obtain
    \begin{eqnarray}
    \label{Eq113AnalyR}
    \|A^{\frac{\beta}{2}-1}w_2\|&\leq &C|\lambda|^{-1}\{\|F\|_\mathcal{H}\|U\|_\mathcal{H}\}^\frac{1}{2}
    \end{eqnarray}
    On the  other hand,  from $w_2=w-w_1$, we  have
   $ \|A^\frac{1}{2} w_2\|\leq \|A^\frac{1}{2} w\|+\|A^\frac{1}{2}w_1\|. $
    From  Lemma \eqref{Lemma01Gevrey} and \eqref{Eq112AnalyR},    we have
 \begin{equation}\label{Eq000Gevrey}
  \|A^\frac{1}{2} w_2\|^2\leq C_\tau\{\|F\|_\mathcal{H}\|U\|_\mathcal{H}+\|F\|^2_\mathcal{H}\}\qquad{\rm for}\qquad 2\leq\theta+2\beta.
    \end{equation}
 Now, by Lions' interpolations inequality ($\frac{\beta}{2}\in[\frac{\beta}{2}-1,\frac{1}{2}]$)  and estimates \eqref{Eq113AnalyR} and \eqref{Eq000Gevrey},  we derive
 \begin{align*}
 \|A^\frac{\beta}{2}w_2|| & \leq  C\|A^{\frac{\beta}{2}-1}w_2\|^\frac{1-\beta}{3-\beta}\|A^\frac{1}{2}w_2\|^\frac{2}{3-\beta}\\
 \leq & C |\lambda|^{-\frac{1-\beta}{3-\beta}}\{\|F\|_\mathcal{H}\|U\|_\mathcal{H}+\|F\|^2_\mathcal{H}\}^\frac{1}{2}\quad{\rm for}\quad 2\leq\theta+2\beta.
 \end{align*}
On the other hand,  as $\|A^\frac{\beta}{2}w\|\leq C\{\|A^\frac{\beta}{2}w_1\|+\|A^\frac{\beta}{2}w_2\|\}$ and $-\frac{1-\beta}{3-\beta}\geq -\frac{1}{2}$,   for $2\leq\theta+2\beta$ we have
 \begin{equation*}
 \|A^\frac{\beta}{2} w\| \leq  C\{ |\lambda|^{-\frac{1-\beta}{3-\beta}}+|\lambda|^{-\frac{1}{2}}\}\{\|F\|_\mathcal{H}\|U\|_\mathcal{H}\}^\frac{1}{2}
 \leq C|\lambda|^{-\frac{1-\beta}{3-\beta}}\{\|F\|_\mathcal{H}\|U\|_\mathcal{H}\}^\frac{1}{2}.
 \end{equation*}
   Equivalently,  
 \begin{equation}\label{Eq001Gevrey}
 |\lambda|\|A^\frac{\beta}{2} w\|^2 \leq C|\lambda|^\frac{1+\beta}{3-\beta}\|F\|_\mathcal{H}\|U\|_\mathcal{H}\qquad{\rm for }\qquad 2\leq\theta+2\beta.
 \end{equation}
 Or still
 \begin{equation}\label{Abetamedio}
 |\lambda|^{\frac{2-2\beta}{3-\beta}}\|w\|\leq C\|F\|_\mathcal{H}\qquad{\rm for }\qquad 2\leq\theta+2\beta.
 \end{equation}

Now, we will estimate the term $|\lambda|\|Au\|^2$.
Taking the duality product between equation 
\eqref{esp-30}  and $\lambda u$ and using 
the equation \eqref{esp-10}, we have
\begin{align*}
        \alpha\lambda \|Au\|^2=&\lambda\dual{(I+\kappa A^\beta)w}{w}+\dual{\lambda(I+\kappa A^\beta)w}{f_1}-\gamma\dual{\dfrac{\lambda}{|\lambda|^\frac{1}{2}}z}{|\lambda|^\frac{1}{2}Au} 
       \\
       &+ \dual{(I+\kappa A^\beta)f_3}{-iw-if_1}\\
        =& \lambda\|w\|^2+\kappa\lambda\|A^\frac{\beta}{2}w\|^2+\dual{i\alpha A^2u+i\gamma Az-i(I+\kappa A^\beta)f_3}{f_1}\\
         -&\gamma\dual{\dfrac{\lambda}{|\lambda|^\frac{1}{2}}z}{|\lambda|^\frac{1}{2}Au} -i\dual{(I+\kappa A^\beta)f_3}{w}+i\dual{(I+\kappa A^\beta)f_3}{f_1} .
    \end{align*}
    Applying Cauchy-Schwarz and Young inequalities, for $\varepsilon>0$, there exists a positive constant $C_\varepsilon$, independent of $\lambda$,  such that:
    \begin{align*}
    |\lambda|\|Au\|^2&\leq  C|\lambda|\|A^\frac{\beta}{2}w\|^2+C\{|\dual{Au}{Af_1}|+|\dual{z}{Af_1}|\}
    \\
    &  +K_\varepsilon|\lambda|\|z\|^2+\varepsilon|\lambda|\|Au\|^2+C\{ |\dual{f_3}{w}|+|\dual{A^\frac{\beta}{2} f_3}{A^\frac{\beta}{2} w}|\}.
    \end{align*}
    Then,  applying Cauchy-Schwarz and Young inequalities,   $\|F\|_\mathcal{H}$ and $\|U\|_\mathcal{H}$, we have
    \begin{equation}\label{Eq001AnalyR1N}
    |\lambda|\|Au\|^2 \leq  C|\lambda|\|A^\frac{\beta}{2}w\|^2+C_\varepsilon|\lambda|\|z\|^2+C_\tau \|F\|_\mathcal{H}\|U\|_\mathcal{H}.
    \end{equation}
   Let's define $\Phi_{s_1}=  \Phi_{s_1}(\theta,\beta)$, as follows
    \begin{equation}\label{Beta+Theta}
    \Phi_{s_1}:=2\max\bigg\{\dfrac{1-\beta}{3-\beta}, \dfrac{\theta}{2+\theta-\beta}\bigg\}\quad{\rm such\quad that}, \quad 2\leq \theta+2\beta.
    \end{equation}
   
      Therefore,  from   \eqref{Beta+Theta},  we have
       \begin{equation}\label{Eq001AnalyRN2}
    |\lambda|^{\Phi_{s_1}}\|Au\| \leq  C_\tau \|F\|_\mathcal{H} \qquad{\rm for }\qquad 2\leq\theta+2\beta.
    \end{equation}

Finally,  we will estimate the term 
$|\lambda|\|A^\frac{1}{2}v\|^2$.
Taking the duality product between equation 
\eqref{esp-30}  and $w$ and using the equation 
\eqref{esp-10}, we have
\begin{align}
\label{Eq001AnalyR}
 i&\lambda [\|w\|^2+\kappa\|A^\frac{\beta}{2}w\|^2-
\alpha\|Au\|^2] \\
\notag
&= -\gamma\langle 
A^\frac{1}{2}z, A^\frac{1}{2}w\rangle +\alpha\langle Au, 
A f_1\rangle +\langle f_3, w\rangle
+\kappa\dual{A^\frac{\beta}
{2}f_3}{ A^\frac{\beta}{2}w}.
\end{align}
Now, taking the duality product between equation \eqref{esp-40}  and $z$ and using the equation \eqref{esp-20}, we have
\begin{multline}\label{Eq002AnalyR}
i\lambda \alpha \| A^\frac{1}{2} v \|^2=i\lambda\|z\|^2+\delta \|A^\frac{\theta}{2}z \|^2-\alpha\langle A^\frac{1}{2} v,A^\frac{1}{2}  f_2\rangle\\
-\gamma\langle A^\frac{1}{2}w, A^\frac{1}{2}z\rangle -\langle f_4, z\rangle.
\end{multline}
Adding the equations \eqref{Eq001AnalyR}  and  \eqref{Eq002AnalyR},    taking the imaginary part  and noting that
$\rm{Im}\{\dual{A^\frac{1}{2}z}{A^\frac{1}{2}w}+\dual{A^\frac{1}{2}w}{A^\frac{1}{2}z}\}=0$, we obtain
\begin{align}\nonumber
 \lambda\alpha\|A^\frac{1}{2}v\|^2  = & \rm{Im} \{ \alpha \langle Au, Af_1\rangle+ \langle f_3, w\rangle+\kappa\dual{A^\frac{\beta}{2}f_3}{ A^\frac{\beta}{2}.w}-\alpha \langle A^\frac{1}{2}v, A^\frac{1}{2}f_2 \rangle\\
 \label{Eq003AnalyR}
 &-\langle f_4,z\rangle \}+\lambda\alpha\|Au\|^2 +\lambda[\|z\|^2-\|w\|^2-\kappa\|A^\frac{\beta}{2}w\|^2]
\end{align}
Otherwise,  now applying Cauchy-Schwarz and Young inequalities in \eqref{Eq003AnalyR}, using estimates   \eqref{Eq119AnalyRR}, \eqref{Eq001Gevrey}   and \eqref{Eq001AnalyRN2}, we find in $R_G$: 
    \begin{equation}\label{Eq104AEAnaly}
    |\lambda|^{\Phi_{s_1}}\|A^\frac{1}{2}v\| \leq  C_\tau\|F\|_\mathcal{H}\qquad\text{for}\qquad 2\leq \theta+2\beta.
    \end{equation}
    Finally,  from  the  estimates \eqref{Eq119AnalyRR}, \eqref{Eq001Gevrey}, \eqref{Eq001AnalyRN2} and  \eqref{Eq104AEAnaly},   for $0<\beta <1$, $0\leq\theta \leq 1$,  we find
    \begin{equation}\label{ClaseGevrey}
    |\lambda|^{\Phi_{s_1}}\|U\|_\mathcal{H}\leq C \|F\|_\mathcal{H}\qquad\text{for}\qquad 2\leq\theta+2\beta,
    \end{equation}
 in the {\it Gevrey Region} $R_{G1}$.  Thus,  the proof of this theorem is concluded.
 \end{proof}

\subsubsection{Gevrey Class  $s_2>\frac{1}{\phi_2}=\frac{2(2+\theta-\beta)}{\theta}$  for $\frac{3\theta}{4}+\frac{\beta}{2}\leq 1$}
\label{3.2.2}

\begin{center}
    \begin{tikzpicture}[scale=3, rotate=0]
    \label{Figura07}
    \draw[thick,->] (0,0) -- (1.25,0);
    \draw[thick,->] (0,0) -- (0,1.25);
    \foreach \x in {0,1}
    \draw (\x cm,1pt) -- (\x cm,-1pt) node[anchor=north] {$\x$};
    \foreach \y in {0,1}
  \draw[thick,->] (0,0) -- (1.25,0)
    node[anchor=north west] {$\theta$ axis};
    \draw[thick,->] (0,0) -- (0,1.25) node[anchor=south east] {$\beta$ axis};
     \draw(0.56,1.1) node[anchor=center] {$\qquad(\frac{2}{3},1)$};
 \draw(1.08,0.54) node[anchor=center] {$\qquad(1, \frac{1}{2})$};
        \draw (-0.1, 1) node[anchor=center] {$1$};
    \coordinate[label=left:] (a) at (0,0);
    \coordinate[label=right:] (b) at (0,1);
    \coordinate[label=left:] (c) at (2/3,1);
    \coordinate[label=left:] (d) at (1,1/2);
    \coordinate[label=left:] (e) at (1,0);
    \draw[rectangle, fill=black!30!white] (a)  -- (b) --  (c) -- (d) -- (e) -- (a) -- cycle;
   \draw[black!25!black, ultra thick] (0, 1) -- (2/3, 1);
   \draw[black!25!black, ultra thick] (1,1/2) -- (2/3, 1);
     \draw[black!25!black, ultra thick] (1, 0) -- (1, 1/2);
 \draw[fill=black!95!black, dashed] (2/3,1) circle (0.02);
    \draw[fill=black!95!black, dashed] (1,1/2) circle (0.02);
      \draw[fill=black!95!black, dashed] (0,1) circle (0.02);
     \draw[fill=white!95!black, dashed] (0,0) circle (0.02);
      \draw[fill=white!95!black, dashed] (1,0) circle (0.02);
  
    \end{tikzpicture}

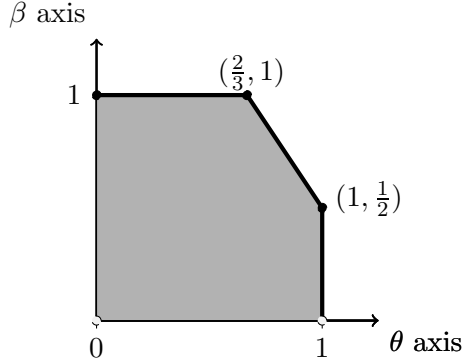
    {\bf Fig.  07}: Gevrey  region 
    $G_{R2}:=\{ (\theta,\beta)\in R_G/ \frac{3\theta}{4}
    +\frac{\beta}{2}\leq 1\}-\{(0,0),(1,1)\} $.
\end{center}

Our main result in this subsubsection
is as  follows:

\begin{theorem} The linear operator $\mathbb{B}$  is the generator of a strongly continuos semigroups of contractions on the Hilbert space $ \mathcal{H}$, associated with the system \eqref{ISplacas-20}-\eqref{ISplacas-30} the strongly continuos   semigroups $S(t)=e^{\mathbb{B}t}$ is  of Gevrey class $s_2$ for every $s_2>\dfrac{1}{\phi_2}=\frac{2(2+\theta-\beta)}{\theta}$ for $\phi_2\in(0,1)$,  as there exists a positive constant $C$ such that we have the resolvent estimate:
    \begin{equation}\label{Eq1.6Tebon2020B}
   \lim\limits_{|\lambda|\to\infty} \sup  |\lambda |^\frac{\theta}{2(2+\theta-\beta)} \|(i\lambda I-\mathbb{B})^{-1}\|_{\mathcal{L}( \mathcal{H})} < \infty, \quad{\rm for} \quad \frac{3\theta}{4}+\frac{\beta}{2}\leq 1.
    \end{equation}
\end{theorem}

\begin{proof}[\bf Proof.]
The proof of this theorem will be carried 
out by estimating through the next lemmas in  each
term of  $|\lambda|\|U\|^2_\mathcal{H}$,  
as follows:

 The following lemma estimates the term 
 $|\lambda|\|A^\frac{\beta}{2}w\|^2$:
 
\begin{lemma}\label{Lemma13A}
Let $\tau > 0$. There exists a positive constant $C_\tau$ such that the solutions of Eqs. \eqref{ISplacas-20}--\eqref{ISplacas-40}, for $|\lambda| \geq \tau$, satisfy the following inequalities:
\begin{equation}\label{Alternativa002}
|\lambda|\|A^\frac{\beta}{2}w\|^2\leq C_\varepsilon |\lambda| \|A^\frac{\theta}{4}z\|^2+C_\tau\{\|F\|_\mathcal{H}\|U\|_\mathcal{H}+\|F\|^2_\mathcal{H}\},
\end{equation}
 for  $\frac{3\theta}{4}+\frac{\beta}{2}\leq 1$.
\end{lemma}

\begin{proof}[\bf Proof.]
Now,    taking the duality product between equation \eqref{esp-30}  and  $w$, using the self-adjointness of the powers of the operator $A$,   we have 
\begin{equation}\label{Eq03GevreyBeta1}
i\lambda\{\|w\|^2+\kappa\|A^\frac{\beta}{2}w\|^2\}=-\alpha\dual{Au}{Aw}-\gamma\dual{z}{A w}+\dual{(I+\kappa A^\beta)f_3}{w}.
\end{equation}
Now, taking the duality product between  $\frac{\alpha}{\gamma}Au$ and  \eqref{esp-40},   we obtain
\begin{align}
\nonumber
-\alpha\dual{Au}{Aw} =&-\dfrac{\alpha}{\gamma}\dual{Au}{i\lambda z}-\dfrac{\alpha^2}{\gamma}\dual{Au}{Av}-\dfrac{\alpha\delta}{\gamma}\dual{Au}{A^ \theta z}+\dfrac{\alpha}{\gamma}\dual{Au}{f_4}\\
\notag
=& \dfrac{\alpha}{\gamma} \dual{Aw }{z}+\dfrac{\alpha}{\gamma}\dual{Af_1}{z}-\dfrac{\alpha^2}{\gamma}\dual{A^2u}{v} -\dfrac{\alpha\delta}{\gamma}\dual{A^2u}{A^{\theta-1} z}\\
\label{Eq04GevreyBeta1}
+&\dfrac{\alpha}{\gamma} \dual{Au}{f_4}.
\end{align}
Using \eqref{Eq04GevreyBeta1} in \eqref{Eq03GevreyBeta1},  we have
\begin{align}
\nonumber
i\lambda\{\|w\|^2+\kappa\|A^\frac{\beta}{2}w\|^2\} =& \dfrac{\alpha}{\gamma} \dual{Aw }{z}+\dfrac{\alpha}{\gamma}\dual{Af_1}{z}-\dfrac{\alpha^2}{\gamma}\dual{A^2u}{v} \\
\label{Eq05GevreyBeta1}
- &\dfrac{\alpha\delta}{\gamma}\dual{A^2u}{A^{\theta-1} z}
+\dfrac{\alpha}{\gamma} \dual{Au}{f_4}-\gamma\dual{z}{A w} \\
\notag
+& \dual{f_3}{w}+\kappa\dual{A^\frac{\beta}{2}f_3}{A^\frac{\beta}{2}w}.
\end{align}
From \eqref{esp-30},  we have  $\alpha A^2u =-i\lambda(I+\kappa A^\beta)w-\gamma Az+(I+\kappa A^\beta)f_3$,   then
\begin{align}
\nonumber
\alpha\dual{A^2u} {v}= & \dual{w}{z}+\dual{w}{f_2}+\kappa\dual{A^\frac{\beta}{2}w}{A^\frac{\beta}{2}z}+\kappa\dual{A^\frac{\beta}{2}w}{A^\frac{\beta}{2}f_2}-\gamma\dual{z}{Av}\\
\notag
+ & \dual{f_3}{v} +\kappa\dual{A^\frac{\beta}{2} f_3}{A^\frac{\beta}{2}v}\\
\nonumber
= & \dual{w}{z}+\dual{w}{f_2}+\kappa\dual{A^\frac{\beta}{2}w}{A^\frac{\beta}{2}z}+\kappa\dual{A^\frac{\beta}{2}w}{A^\frac{\beta}{2}f_2}-\dfrac{\gamma}{\alpha}i\lambda\|z\|^2 \\
\label{Eq06GevreyBeta1}
-& \dfrac{\gamma^2}{\alpha}\dual{z}{Aw}+\dfrac{\delta\gamma}{\alpha}\|A^\frac{\theta}{2}z\|^2-\dfrac{\gamma}{\alpha}\dual{z}{f_4}\\
\notag
+ & \dual{f_3}{v} +\kappa\dual{A^\frac{\beta}{2} f_3}{A^\frac{\beta}{2}v},
\end{align}
and
\begin{align}
\nonumber
\alpha\dual{A^2u} {A^{\theta-1}z}= & -i\dual{\frac{\lambda}{\sqrt{|\lambda|}}A^{\theta-1} w}{\sqrt{|\lambda|}z}-i\kappa\dual{\frac{\lambda}{\sqrt{|\lambda|}}A^\frac{\beta}{2}w}{\sqrt{|\lambda|}A^{\theta+\frac{\beta}{2}-1}z}\\
\label{Eq07GevreyBeta1}
- & \gamma\|A^\frac{\theta}{2} z\|^2+\dual{f_3}{A^{\theta-1}z}+\kappa\dual{A^\frac{\beta}{2}f_3}{A^{\theta+\frac{\beta}{2}-1}z}.
\end{align}
Now,  using \eqref{Eq06GevreyBeta1} and \eqref{Eq07GevreyBeta1} in \eqref{Eq05GevreyBeta1},  we get
 \begin{align*}
i\lambda\{\|w\|^2+\kappa\|A^\frac{\beta}{2}w\|^2\} =&
 \dfrac{\alpha}{\gamma}\dual{Af_1}{z}-\dfrac{\alpha}{\gamma}\dual{w}{z}-\dfrac{\alpha}{\gamma}\dual{w}{f_2}-\dfrac{\alpha\kappa}{\gamma}\dual{A^\beta w}{z}\\
- &\dfrac{\alpha\kappa}{\gamma}\dual{A^\frac{\beta}{2}w}{A^\frac{\beta}{2}f_2}+i\lambda\|z\|^2+\gamma\dual{z}{Aw}+\dual{z}{f_4} \\
\nonumber
 - & \delta\|A^\frac{\theta}{2}z\|^2-\dfrac{\alpha\kappa}{\gamma}\dual{A^\frac{\beta}{2} f_3}{A^\frac{\beta}{2}v}-\dfrac{\delta}{\gamma}\dual{f_3}{A^{\theta-1}z}\\
-&\dfrac{\alpha}{\gamma}\dual{f_3}{v}+ i\dfrac{\delta\kappa}{\gamma}\dual{\frac{\lambda}{\sqrt{|\lambda|}}A^\frac{\beta}{2}w}{\sqrt{|\lambda|}A^{\theta+\frac{\beta}{2}-1}z}\\
+ & \dfrac{\alpha}{\gamma} \dual{Au}{f_4}-\gamma\dual{z}{A w}+\dual{f_3}{w}+\kappa\dual{A^\frac{\beta}{2}f_3}{A^\frac{\beta}{2}w}\\
+ & \dfrac{\alpha}{\gamma} \dual{Aw }{z}+\delta\|A^\frac{\theta}{2} z\|^2 -\dfrac{\delta\kappa}{\gamma}\dual{A^\frac{\beta}{2}f_3}{A^{\theta+\frac{\beta}{2}-1}z}\\
+ & i\dfrac{\delta}{\gamma}\dual{\frac{\lambda}{\sqrt{|\lambda|}}A^{\theta-1} w}{\sqrt{|\lambda|}z}
\end{align*}
Then
\begin{align}
\nonumber
i\lambda\{\|w\|^2+\kappa\|A^\frac{\beta}{2}w\|^2\} =& \dfrac{\alpha}{\gamma} \dual{Aw }{z}+\dfrac{\alpha}{\gamma}\dual{Af_1}{z}-\dfrac{\alpha}{\gamma}\dual{w}{z}-\dfrac{\alpha}{\gamma}\dual{w}{f_2}\\
\nonumber
-& \dfrac{\alpha\kappa}{\gamma}\dual{A^\beta w}{z}-\dfrac{\alpha\kappa}{\gamma}\dual{A^\frac{\beta}{2}w}{A^\frac{\beta}{2}f_2}+\dual{z}{f_4}\\
 \nonumber
 + &i\dfrac{\delta}{\gamma}\dual{\frac{\lambda}{\sqrt{|\lambda|}}A^{\theta-1} w}{\sqrt{|\lambda|}z} -\dfrac{\delta}{\gamma}\dual{f_3}{A^{\theta-1}z} \\
\label{Eq08GevreyBeta1}
 & +\dual{f_3}{w}+i\dfrac{\delta\kappa}{\gamma}\dual{\frac{\lambda}{\sqrt{|\lambda|}}A^\frac{\beta}{2}w}{\sqrt{|\lambda|}A^{\theta+\frac{\beta}{2}-1}z}\\
 \nonumber
 - &\dfrac{\alpha\kappa}{\gamma}\dual{A^\frac{\beta}{2} f_3}{A^\frac{\beta}{2}v}+i\lambda\|z\|^2
 -\dfrac{\alpha}{\gamma}\dual{f_3}{v}+\dfrac{\alpha}{\gamma} \dual{Au}{f_4} \\
\nonumber
&-\dfrac{\delta\kappa}{\gamma}\dual{A^\frac{\beta}{2}f_3}{A^{\theta+\frac{\beta}{2}-1}z}+\kappa\dual{A^\frac{\beta}{2}f_3}{A^\frac{\beta}{2}w}.
\end{align}
As  
\begin{align*}
\frac{\alpha}{\gamma}\dual{Aw}{z}= & i\frac{\alpha}{\gamma}\dual{\frac{\lambda}{\sqrt{|\lambda|}}Au}{\sqrt{|\lambda|}z}-\frac{\alpha}{\gamma}\dual{Af_1}{z}, 
\end{align*} 
$-\gamma\dual{z}{Aw}=-i\gamma\dual{\sqrt{|\lambda|}z}{\frac{\lambda}{\sqrt{|\lambda|}}Au}+\gamma\dual{z}{Af_1}$  and 
similarly  $-\dfrac{\alpha\kappa}{\gamma}\dual{A^\beta w}{z}=-i\dfrac{\alpha\kappa}{\gamma}\dual{\frac{\lambda}{\sqrt{|\lambda|}}A^\beta u}{\sqrt{|\lambda|}z}+\dfrac{\alpha\kappa}{\gamma}\dual{A^\beta f_1}{z}$.
Taking the imaginary part,  using the norms $\|U\|_\mathcal{H}^2, \|F\|^2_\mathcal{H}$ and from  $\theta-1\leq\frac{\beta}{2}$,    for $\varepsilon>0$ exists $C_\varepsilon>0$,   such that
\begin{align}
\nonumber
|\lambda|\{\|w\|^2+\kappa\|A^\frac{\beta}{2}w\|^2\} \leq &  \varepsilon|\lambda|\|Au\|^2+\varepsilon|\lambda|\|A^\beta u\|^2 -\dfrac{\delta\kappa}{\gamma}{\rm Im}\dual{A^\frac{\beta}{2}f_3}{A^{\theta+\frac{\beta}{2}-1}z}\\
\nonumber
+ & C_\varepsilon|\lambda|\|z\|^2+
C\{\|w\|^2+\|z\|^2 \}+C_\tau\|F\|_\mathcal{H}\|U\|_\mathcal{H}\\
\label{Eq09GevreyBeta1}
 +& \dfrac{\delta\kappa}{\gamma}{\rm Re}\dual{\frac{\lambda}{\sqrt{|\lambda|}}A^\frac{\beta}{2}w}{\sqrt{|\lambda|}A^{\theta+\frac{\beta}{2}-1}z}  .
\end{align}
As $\frac{3\theta}{4}+\frac{\beta}{2}\leq 1$ in particular we will have   $\theta+\frac{\beta}{2}-1\leq\frac{\theta}{2}$, then  using \eqref{dis-10}  and from $\|F\|^2_\mathcal{H}$,  we have  $  |-\dfrac{\delta\kappa}{\gamma}{\rm Im}\dual{A^\frac{\beta}{2}f_3}{A^{\theta+\frac{\beta}{2}-1}z}|\leq C_\tau\{\|F\|_\mathcal{H}\|U\|_\mathcal{H}+\|F\|_\mathcal{H}^2\}$. 
Now, as 
 $\beta\leq 1$ using continuous embedding  and  considering   $\theta+\frac{\beta}{2}-1\leq \frac{\theta}{4}$  $\Longleftrightarrow \frac{3\theta}{4}+\frac{\beta}{2}\leq 1$,   for $\varepsilon>0$, exists  $C_\varepsilon>0$,  such that: for  $\frac{3\theta}{4}+\frac{\beta}{2}\leq 1$, we get
\begin{equation}\label{Alternativa000}
|\lambda|\|A^\frac{\beta}{2}w\|^2  \leq  C_\varepsilon |\lambda| \|A^\frac{\theta}{4}z\|^2+C_\tau\{\|F\|_\mathcal{H}\|U\|_\mathcal{H}+\|F\|^2_\mathcal{H}\} +\varepsilon|\lambda|\|Au\|^2.
\end{equation}
Furthermore,  from estimate \eqref{Eq001AnalyR1N}  for $0\leq\theta\leq 1$ and $0<\beta\leq 1$,  we have
\begin{eqnarray}\label{Alternativa001}
\varepsilon|\lambda|\|Au\|^2&\leq &\varepsilon |\lambda|\|A^\frac{\beta}{2}w\|^2+\varepsilon\|F\|_\mathcal{H} \|U\|_\mathcal{H} +C_\varepsilon|\lambda|\|z\|^2.
\end{eqnarray}
Using \eqref{Alternativa001}  in \eqref{Alternativa000},  we finish the proof of the present lemma.
\end{proof}

 The following lemma estimates the term 
 $|\lambda|\|A^\frac{\theta}{4}z\|^2 $.

\begin{lemma}\label{Lemma013}
Let 
$\tau > 0$. There exists a positive constant 
$C_\tau$ such that the solutions of 
Eqs. \eqref{ISplacas-20}--\eqref{ISplacas-40}, 
for 
$|\lambda| \geq \tau$,  
satisfy the following inequalities:
    \begin{equation}\label{Eq123AnalyRRG2}
 \hspace*{-1cm}   |\lambda|\|A^\frac{\theta}{4}z\|^2\leq C_\tau|\lambda|^\frac{2-\beta}{2+\theta-\beta}\{ \|F\|_\mathcal{H}\|U\|_\mathcal{H}+\|F\|^2_\mathcal{H}\}, 
    \end{equation}
    for \quad $0\leq\theta\leq 1$\quad and \quad $0<\beta\leq 1$.
\end{lemma}

\begin{proof}[\bf Proof.]
Henceforth, we assume $\lambda\in\R$ with  $|\lambda|>1$.   Now,  let us decompose $z$ as  $z=z_1+z_2$, where $z_1\in D(A)$ and $z_2\in D(A^0)$, with
    \begin{equation}\label{Eq110AnalyRRG2}
    i\lambda z_1+Az_1=f_4, \qquad i\lambda z_2=-\alpha Av+\gamma Aw-\delta A^\theta z+Az_1.
    \end{equation}
    Firstly,  applying the duality product  of the first equation in \eqref{Eq110AnalyRRG2} by $A^\frac{\theta}{2}z_1$, we have
    \begin{equation}\label{Eq111AnalyRRG2}
    i\lambda\|A^\frac{\theta}{4}z_1\|^2+\|A^\frac{2+\theta}{4}z_1\|^2=\dual{f_4}{A^\frac{\theta}{2}z_1}.
    \end{equation}
    Taking first the imaginary part of \eqref{Eq111AnalyRRG2} and in the sequence the real part and applying Cauchy-Schwarz inequality, we have
    \begin{multline}\label{Eq112AnalyRRG2}
   |\lambda|\|A^\frac{\theta}{4}z_1\|^2=|\rm{Im}\dual{f_4}{A^\frac{\theta}{2}z_1}|\leq C\|F\|_\mathcal{H}\|A^\frac{\theta}{2} z_1\|, \\
 {\rm and}\qquad   \|A^\frac{2+\theta}{4}z_1\|^2\leq C\|F\|_\mathcal{H}\|A^\frac{\theta}{2}z_1\|.
    \end{multline}
  From second equation in \eqref{Eq112AnalyRRG2},   for $\varepsilon>0$ exists  $C_\varepsilon>0$ such that
    \begin{equation}\label{Eq113AnalyRRG2}
    \|A^\frac{2+\theta}{4}z_1\|^2\leq C_\varepsilon\|F\|^2_\mathcal{H}
    +\varepsilon\|A^\frac{\theta}{2}z_1\|^2\;\Longleftrightarrow\;  \|A^\frac{2+\theta}{4}z_1\|\leq C_\varepsilon\|F\|_\mathcal{H}.
    \end{equation}
    As $\frac{\theta}{2}\leq\frac{2+\theta}{4}$,  applying  continuous embedding and estimate \eqref{Eq113AnalyRRG2} in first inequality of \eqref{Eq112AnalyRRG2}, we obtain
    \begin{equation}\label{Eq114AnalyRRG2}
     |\lambda|\|A^\frac{\theta}{4}z_1\|^2\leq C\|F\|^2_\mathcal{H}.
    \end{equation}   
    Equivalently, 
    \begin{multline}\label{Eq116AnalyRRG2}
   \|A^\frac{\theta}{4}z_1\|\leq C\dfrac{ \{\|F\|_\mathcal{H}\|U\|_\mathcal{H}+ \|F\|^2_\mathcal{H}\}^\frac{1}{2}}{|\lambda|^\frac{1}{2}},\\
   {\rm and}\qquad \|A^\frac{2+\theta}{4}z_1\|\leq C\{\|F\|_\mathcal{H}\|U\|_\mathcal{H}+ \|F\|^2_\mathcal{H}\}^\frac{1}{2}.
    \end{multline}
  In what follows from the second equation in \eqref{Eq110AnalyRRG2} that
    \begin{equation*}
    i\lambda A^\frac{\beta-2}{2}z_2=-\alpha A^\frac{\beta}{2}v+\gamma A^\frac{\beta}{2}w-\delta A^{\theta+\frac{\beta}{2}-1} z+ A^\frac{\beta}{2} z_1,
    \end{equation*}
   then,  from \eqref{ExponentialP1},  for   $|\lambda|>1$ and   since $\frac{\beta}{2}\leq \frac{1}{2},  \theta+\frac{\beta}{2}-1\leq\frac{\theta}{2}$ and $\frac{\beta}{2}\leq\frac{2+\theta}{4}$,  applying  continuous embedding and using first inequality of \eqref{Eq116AnalyRRG2},  estimate \eqref{dis-10} and second inequality of  \eqref{Eq116AnalyRRG2},   for $0\leq\theta\leq 1$ and $0<\beta\leq 1$,  we obtain
    \begin{equation}\label{Eq117AnalyRRG2}
    \|A^\frac{\beta-2}{2}z_2\|\leq C_\tau|\lambda|^{-1}\{\|F\|_\mathcal{H}\|U\|_\mathcal{H}+\|F\|^2_\mathcal{H}\}^\frac{1}{2}.
    \end{equation}
    On the  other hand, from $z_2=z-z_1$,  using \eqref{dis-10} and as $\frac{\theta}{2}\leq\frac{1}{2}$ applying  continuous embedding and the second  inequality of \eqref{Eq112AnalyRR},  we  have
    \begin{equation}  \label{Eq118AnalyRRG2}
    \|A^\frac{\theta}{2} z_2\| \leq  \|A^\frac{\theta}{2} z\|+\|A^\frac{\theta}{2}z_1\|
  \leq  C\{\|F\|_\mathcal{H}\|U\|_\mathcal{H}+\|F\|_\mathcal{H}^2\}^\frac{1}{2}.
    \end{equation}
    Now,  by Lions' interpolations inequality for  $\frac{\theta}{4}\in[\frac{\beta-2}{2} , \frac{\theta}{2}]$,  we derive
    \begin{equation}\label{Eq119AnalyRRG2}
    |\lambda|^\frac{1}{2} \|A^\frac{\theta}{4}z_2\|\leq C|\lambda|^\frac{1}{2}\|A^\frac{\beta-2}{2}z_2\|^\frac{\theta}{2(2+\theta-\beta)}\|A^\frac{\theta}{2}z_2\|^\frac{\theta-2\beta+4}{2(2+\theta-\beta)}.
    \end{equation}
    From \eqref{Eq117AnalyRRG2} and $|\lambda |\geq 1$,  we have
    \begin{equation}\label{Eq120AnalyRRG2}
    \|A^\frac{\beta-2}{2}z_2\|^\frac{\theta}{2(2+\theta-\beta)} \leq C_\tau |\lambda|^{-\frac{\theta}{2(2+\theta-\beta)}}\{\|F\|_\mathcal{H}\|U\|_\mathcal{H}+\|F\|^2_\mathcal{H}\}^\frac{\theta}{4(2+\theta-\beta)},
    \end{equation}
 and from \eqref{Eq118AnalyRRG2}   we have
    \begin{equation}\label{Eq121AnalyRRG2}
    \|A^\frac{\theta}{2} z_2\|^\frac{\theta+4}{2(2+\theta)}\leq C_\tau \{\|F\|_\mathcal{H}\|U\|_\mathcal{H}+\|F\|^2_\mathcal{H}\}^\frac{\theta+4}{4(2+\theta)}.
    \end{equation}
    Then, using \eqref{Eq120AnalyRRG2} and \eqref{Eq121AnalyRRG2} in \eqref{Eq119AnalyRRG2}, for $ |\lambda|>1$, we derive
    \begin{equation}\label{Eq122AnalyRRG2}
    |\lambda|^\frac{1}{2} \|A^\frac{\theta}{4}z_2\|\leq C_\tau |\lambda|^{\frac{2-\beta}{2(2+\theta-\beta)}}\{ \|F\|_\mathcal{H}\|U\|_\mathcal{H}+\|F\|^2_\mathcal{H}\}^\frac{1}{2}.
    \end{equation}
    Therefore, from $A^\frac{\theta}{4}z=A^\frac{\theta}{4}z_1+A^\frac{\theta}{4}z_2$ and  first inequality of  \eqref{Eq116AnalyRRG2} and \eqref{Eq121AnalyRRG2}, we complete the proof of this lemma.
\end{proof}

\begin{lemma}\label{Lemma014}
Let $\tau > 0$. There exists a positive constant $C_\tau$ such that the solutions of Eqs. \eqref{ISplacas-20}--\eqref{ISplacas-40}, for $|\lambda| \geq \tau$, satisfy the following inequalities:
\begin{equation} \label{Eq124AnalyRRG2}
|\lambda|\|Au\|^2\leq  C_\tau|\lambda|^\frac{2-\beta}{2+\theta-\beta}\{\|F\|_\mathcal{H} \|U\|_\mathcal{H} +\|F\|^2_\mathcal{H}\}\quad {\rm for} \quad \dfrac{3\theta}{4}+\dfrac{\beta}{2}\leq 1.
\end{equation}
\end{lemma}

\begin{proof}[\bf Proof.]
For $0\leq\theta\leq 1$ and $0<\beta\leq 1$,   applying Lemmas  \eqref{Lemma13A} and \eqref{Lemma013}  in   \eqref{Eq001AnalyR1N},  we finish the proof of the present lemma.
\end{proof}

\begin{lemma}\label{Lemma015}
Let $\tau > 0$. There exists a positive constant $C_\tau$ such that the solutions of equations  \eqref{ISplacas-20}--\eqref{ISplacas-40},  for $|\lambda| \geq \tau$, satisfy the following inequalities:
\begin{equation} \label{Eq125AnalyRRG2}
|\lambda|\|A^\frac{1}{2}v\|^2\leq  C_\tau|\lambda|^\frac{2-\beta}{2+\theta-\beta}\{\|F\|_\mathcal{H} \|U\|_\mathcal{H} +\|F\|^2_\mathcal{H}\}\quad {\rm for} \quad \dfrac{3\theta}{4}+\dfrac{\beta}{2}\leq 1.
\end{equation}
\end{lemma}

\begin{proof}[\bf Proof.]
Taking the product of duality between the equation \eqref{esp-40}  and $z$,  and using estimates  \eqref{esp-10}, \eqref{esp-20} and \eqref{esp-30}, we have
\begin{align*}
i\alpha\lambda\|A^\frac{1}{2} v\|^2  = &-\alpha\dual{ A^\frac{1}{2}v}{A^\frac{1}{2} f_2}+i\lambda\|z\|^2-\dual{ w}{\gamma Az}+\delta\|A^\frac{\theta}{2} z\|^2-\dual{f_4}{ z}\\
&= -\alpha\dual{  A^\frac{1}{2}v}{A^\frac{1}{2} f_2}+i\lambda\|z\|^2-\dual{ w}{-i\lambda(I+\kappa A^\beta)w}\\
&  -\dual{w}{-\alpha A^2 u}-\dual{ w}{(I+\kappa A^\beta)f_3}+\delta\|A^\frac{\theta}{2}z\|^2-\dual{f_4}{ z}\\
& =  i\lambda[\|z\|^2-\|w\|^2-\kappa \|A^\frac{\beta}{2} w\|^2+\alpha\|Au\|^2]-\alpha\dual{Af_1}{ Au}\\
&  -\alpha\dual{ A^\frac{1}{2}v}{A^\frac{1}{2} f_2}-\dual{ w}{(I+\kappa A^\beta)f_3}+\delta\|A^\frac{\theta}{2}z\|^2-\dual{f_4}{ z}.
\end{align*}
Then,  taking imaginary part,   applying   Cauchy-Schwartz, Young inequalities, continuous embedding and estimative \eqref{dis-10A},   for $\varepsilon>0$, exist $C_\varepsilon$,  such that
\begin{equation*}
|\lambda|\|A^\frac{1}{2} v\|^2  \leq C_\tau |\lambda|\{\|z\|^2+\|Au\|^2\}+ C_\tau\{\|F\|_\mathcal{H}\|U\|_\mathcal{H}\}, 
\end{equation*} for\quad $0\leq\theta\leq 1$ and  $0<\beta\leq 1$.

From $0\leq\frac{\theta}{4}$,   using continuous embedding and Lemmas \eqref{Lemma013} and \eqref{Lemma014},   we finish the proof of the present lemma.
\end{proof}
Therefore,   from Lemmas \eqref{Lemma13A}, \eqref{Lemma013},  \eqref{Lemma014} and \eqref{Lemma015},  we arrive at
\begin{equation}\label{GevreyRG2}
|\lambda|\|U\|_\mathcal{H}^2\leq C_\tau|\lambda|^\frac{2-\beta}{2+\theta-\beta}\{\|F\|_\mathcal{H}\|U\|_\mathcal{H}+\|F\|_\mathcal{H}^2\}\quad{\rm for}\quad \dfrac{3\theta}{4}+\dfrac{\beta}{2}\leq 1.
\end{equation}
Equivalently,   for $\varepsilon>0$ exists $C_\varepsilon>0$ such that
\begin{equation*}
|\lambda|^\frac{\theta}{2+\theta-\beta}\|U\|^2_\mathcal{H}\leq \varepsilon\|U\|^2_\mathcal{H}+C_\varepsilon\|F\|^2_\mathcal{H}\qquad{\rm for}\qquad \dfrac{3\theta}{4}+\dfrac{\beta}{2}\leq 1.
\end{equation*}
Then,   for $|\lambda|>1$,  we have
\begin{multline}\label{GevreyRG2B}
|\lambda|^\frac{\theta}{2(2+\theta-\beta)}\|(i\lambda I-\mathbb{B})^{-1}\|_{\mathcal{L}(\mathcal{H})}\leq |\lambda|^\frac{\theta}{2(2+\theta-\beta)}\dfrac{\|U\|_\mathcal{H}}{\|F\|_\mathcal{H}}\\
\leq C_\tau\quad{\rm for}\quad \dfrac{3\theta}{4}+\dfrac{\beta}{2}\leq 1.
\end{multline}
For $\phi_2=\frac{\theta}{2(2+\theta-\beta)}$.  Applying the upper limit as $|\lambda|\to\infty$ in \eqref{GevreyRG2B} completes the proof of this theorem.
\end{proof}

\subsection{Analyticity  of $S(t)=e^{\mathbb{B}t}$ for $(\theta,\beta)=(\frac{1}{2},1)$}
\label{3.3}
To study the analyticity of $S(t)=e^{\mathbb{B}t}$ we will use a characterization given in the book by Liu-Zheng \cite{LiuZ}, Theorem \ref{LiuZAnaliticity}.
From stationary system $(i\lambda I- \mathbb{B})U = F$,  for $(\theta,\beta)=(\frac{1}{2},1)$ in \eqref{esp-10}-\eqref{esp-40},  we have
\begin{eqnarray}
i\lambda u-w &=& f_1\label{esp-10An}\\
i\lambda v-z &=& f_2\label{esp-20An}\\
i\lambda (I+\kappa A) w+ \alpha A^2 u+\gamma Az &=&(I+\kappa A)f_3\label{esp-30An}\\
i\lambda  z+\alpha A v-\gamma Aw+\delta A^\frac{1}{2} z&=& f_4.\label{esp-40An}
\end{eqnarray}
And  from \eqref{dis-10} for $\theta=\frac{1}{2}$,  we have
\begin{eqnarray}\label{dis-10An}
\delta\|A^{\frac{1}{4}}z\|^2=\text{Re}\dual{(i\lambda -\mathbb{B})U}{U}_\mathcal{H}=\text{Re}\dual{F}{U}_\mathcal{H}\leq \|F\|_\mathcal{H}\|U\|_\mathcal{H}.
\end{eqnarray}

 \smallskip
 
 \noindent
{\bf Estimate of $\|A^\frac{3-\beta}{2} u\|$.}
Taking the duality product between equation 
\eqref{esp-30} and $A^{1-\beta}u$,  
using 
\eqref{esp-10} and taking advantage of the 
self-adjointness of the powers of the operator 
$A$, we obtain
\begin{align*}
\alpha \|A^\frac{3-\beta}{2}u\|^2 = & \|A^\frac{1-\beta}{2}w\|^2+\kappa\|A^\frac{1}{2}w\|^2-\gamma\dual{A^\frac{\theta}{2}z}{A^{2-\beta-\frac{\theta}{2}} u}\\
& +\dual{f_3}{A^{1-\beta}u} +\kappa\dual{A^\frac{\beta}{2}f_3}{A^{1-\frac{\beta}{2}}u},
\end{align*}
from $\frac{1-\beta}{2}\leq\frac{1}{2}$ and  $2-\beta-\frac{\theta}{2}\leq \frac{3-\beta}{2}$,  Lemma\eqref{Lemma01Gevrey}, estimate \eqref{dis-10} and continuous embedding, we have
\begin{equation}\label{Analitico007}
\|A^\frac{3-\beta}{2} u\|^2\leq C_\tau\|F\|_\mathcal{H}\|U\|_\mathcal{H}, \qquad {\rm for}\quad 2\leq 2\beta+\theta.
\end{equation}

\begin{lemma}\label{Lemma016}
Let $1\geq\theta\geq\frac{1}{2}$ and $\tau>0$.  There exists a constant   $C_\tau>0$ such that the solutions of \eqref{esp-10An}--\eqref{esp-40An}
for $|\lambda|>\tau$ satisfy the inequality
\begin{equation}\label{Analitico008B}
|\lambda|\|A^\frac{1}{2} v\|^2 \leq 
C_\tau \|F\|_\mathcal{H}\|U\|_\mathcal{H}
\quad{\rm for}\quad 2\leq 2\beta+\theta, \quad{\rm and}\quad 1\leq 3\theta-\beta.
\end{equation}
\end{lemma}

\begin{proof}[\bf Proof.]
Taking the duality product between equation \eqref{esp-40} and $ A^{1-\theta}v$,    using the equation \eqref{esp-20} and   taking advantage of the self-adjointness of the powers of the operator $A$,  we obtain
\begin{multline*}
\delta\dual{A^\theta(i\lambda v-f_2)}{A^{1-\theta}v} =
\dual{ A^{1-\theta}z}{i\lambda v}-\alpha \|A^\frac{2-\theta}{2}v\|^2\\
+\gamma\dual{A(i\lambda u+f_1)}{ A^{1-\theta}v}
+\dual{f_4}{ A^{1-\theta}v},
\end{multline*}
then
\begin{align*}
i\delta\lambda\|A^\frac{1}{2}v\|^2 = & \delta\dual{A^\frac{1}{2}f_2}{ A^\frac{1}{2}v}+ \|A^\frac{1-\theta}{2}z\|^2+\dual{ z}{A^{1-\theta}f_2}
-\alpha \|A^\frac{2-\theta}{2}v\|^2\\
& -\gamma\dual{Au}{A^{1-\theta}(i\lambda v)} -\gamma\dual{Af_1}{ A^{1-\theta}v}+\dual{f_4}{ A^{1-\theta}v}\\
=& \delta\dual{A^\frac{1}{2}f_2}{A^\frac{1}{2}v}+ \|A^\frac{1-\theta}{2}z\|^2+\dual{z}{A^{1-\theta}f_2}
-\alpha \|A^\frac{2-\theta}{2}v\|^2\\
& -\gamma\dual{A^\frac{3-\beta}{2}u}{A^\frac{\beta+1-2\theta}{2}z}-\gamma\dual{Au}{A^{1-\theta}f_2}-\gamma\dual{Af_1}{A^{1-\theta}v}\\
&  +\dual{f_4}{ A^{1-\theta}v}.
\end{align*}
Taking imaginary part,  from  norms $\|F\|_\mathcal{H},\|U\|_\mathcal{H}$ and considering $\frac{\beta+1-2\theta}{2}\leq \frac{\theta}{2}\Longleftrightarrow 1\leq 3\theta-\beta$ and $1-\theta\leq\frac{1}{2}\Longleftrightarrow \frac{1}{2}\leq \theta$, using continuous embedding,  we get
\begin{multline*}
|\lambda|\|A^\frac{1}{2}v\|^2 \leq   C\{\|A^\frac{3-\beta}{2}u\|^2+ \|A^\frac{\theta}{2}z\|^2+\|A^\frac{1}{2} f_2\|\|A^\frac{1}{2}v\|+\|z\|\|A^\frac{1}{2}f_2\| \\
 +\|Au\|\|A^\frac{1}{2}f_2\|+\|Af_1\|\|A^\frac{1}{2}v\|  +\|f_4\|\|A^\frac{1}{2}v\|\}.
\end{multline*}
Therefore, using estimates 
\eqref{dis-10},  \eqref{Analitico007}, 
norm  
$\|F\|_\mathcal{H}$ and $\|U\|_\mathcal{H}$,  
we finish to prove this lemma.
\end{proof}

Applying the duality product between \eqref{esp-30An} and $w$,   using the self-adjointness of the powers of the operator $A$, we obtain   
\begin{multline}\label{An001}
i\lambda\{\|w\|^2+\kappa\|A^\frac{1}{2}w\|^2\}  = -\alpha\dual{A^2u}{w}
-\gamma\dual{A^\frac{1}{2}z}{A^\frac{1}{2}w}\\+\dual{f_3}{w}+\kappa\dual{A^\frac{1}{2}f_3}{A^\frac{1}{2}w}.
\end{multline}
Now applying the duality product between \eqref{esp-40An} and $\frac{\alpha}{\gamma} Au$,   using the self-adjointness of the powers of the operator $A$,  from \eqref{esp-10An},  we obtain   
\begin{multline}\label{An002}
i\lambda\dfrac{\alpha}{\gamma}\dual{z}{Au}=-\dfrac{\alpha^2}{\gamma}\dual{Av}{Au}+\alpha\dual{w}{A^2u}-\dfrac{\delta\alpha}{\gamma}\dual{A^\frac{1}{2}z}{Au}\\+\dfrac{\alpha}{\gamma}\dual{f_4}{Au}.
\end{multline}
From identity Im$\{ -\alpha\dual{A^2u}{w}-\alpha\dual{w}{A^2u}\}=0$,   making the subtraction between  equations  \eqref{An001} and \eqref{An002} and taking imaginary part,  we have
\begin{align}
\nonumber
\lambda \{\|w\|^2+\kappa\|A^\frac{1}{2}w\|^2\} = & {\rm Im} \{-\gamma\dual{A^\frac{1}{2}z}{A^\frac{1}{2}w}+\dual{f_3}{w}+\kappa\dual{A^\frac{1}{2}f_3}{A^\frac{1}{2}w}\\
\label{An003}
&-\dfrac{\alpha}{\gamma}\dual{f_4}{Au}-\dfrac{\alpha}{\gamma}\dual{A^\frac{1}{2}z}{A^\frac{1}{2}(i\lambda u)}\\
\notag
& +\dfrac{\alpha}{\gamma}\dual{v}{\alpha A^2u}+\dfrac{\delta\alpha}{\gamma}\dual{A^\frac{1}{2}z}{Au}\}.
\end{align}
In addition, from  equation \eqref{esp-30An},  we have
\begin{align}
\nonumber
\dfrac{\alpha}{\gamma}\dual{v}{\alpha A^2u} = & \dfrac{\alpha}{\gamma}\dual{v}{-i\lambda(I+\kappa A)w-\gamma Az+(I+\kappa A)f_3}\\
\label{An004}
& = \dfrac{\alpha}{\gamma}\dual{z+f_2}{w}+\dfrac{\alpha\kappa}{\gamma}\dual{A^\frac{1}{2}(z+f_2)}{A^\frac{1}{2}w}+\dfrac{\alpha}{\gamma}\dual{v}{f_3}\\
\nonumber &  -\alpha\dual{A^\frac{1}{2}v}{A^\frac{1}{2}z}+\dfrac{\alpha\kappa}{\gamma}\dual{A^\frac{1}{2}v}{A^\frac{1}{2}f_3}.
\end{align}
Using identity \eqref{An004} 
in \eqref{An003} we have
\begin{align}
\nonumber
\lambda \{\|w\|^2+\kappa\|A^\frac{1}{2}w\|^2\} = & {\rm Im}\bigg\{ \dfrac{\alpha\kappa-\alpha-\gamma^2}{\gamma}\dual{A^\frac{1}{2}z}{A^\frac{1}{2}w}+\kappa\dual{A^\frac{1}{2}f_3}{A^\frac{1}{2}w}\\
\nonumber
&  +\dual{f_3}{w}+\dfrac{\alpha}{\gamma}\dual{f_2}{w}-\dfrac{\alpha}{\gamma}\dual{z}{Af_1}+\dfrac{\alpha}{\gamma}\dual{z}{w}\\
\nonumber
&+\dfrac{\kappa\alpha}{\gamma}\dual{A^\frac{1}{2} f_2}{A^\frac{1}{2}w} -\alpha\dual{A^\frac{1}{2} v}{A^\frac{1}{2}(
i\lambda v-f_2)} \\
\nonumber
& +\dfrac{\alpha}{\gamma}\dual{v}{f_3} +\dfrac{\alpha\kappa}{\gamma}\dual{A^\frac{1}{2}v}{A^\frac{1}{2}f_3}+\dfrac{\delta\alpha}{\gamma}\dual{A^\frac{1}{2}z}{Au}\\
\label{An004B}
& -\dfrac{\alpha}{\gamma}\dual{f_4}{Au}\bigg\}.
\end{align}
For $\varepsilon>0$, there exists $C_\varepsilon>0$ such that $|\dual{A^\frac{1}{2}z}{A^\frac{1}{2}w}|\leq\varepsilon\|A^\frac{1}{2}w\|^2+C_\varepsilon\|A^\frac{1}{2}z\|^2$.   Applying   Cauchy-Schwarz and Young inequalities,  Lemma \eqref{Lemma016} for $(\theta,\beta)=(\frac{1}{2},1)$ and estimate \eqref{dis-10An}  in \eqref{An004B}, we have
\begin{equation}\label{An005}
|\lambda|\{\|w\|^2+\kappa\|A^\frac{1}{2}w\|^2\}\leq C_\tau\|F\|_\mathcal{H}\|U\|_\mathcal{H}\qquad {\rm for}\quad (\theta,\beta)=(\frac{1}{2},1).
\end{equation}

 \smallskip
 
 \noindent
{\bf Estimate of $\alpha|\lambda|\|Au\|^2$.}
Using \eqref{esp-10An} in \eqref{An001},  
we have
\begin{align*}
i\alpha\lambda\|Au\|^2 = & -\alpha\dual{Au}{Af_1}+i\lambda \{\|w\|^2+\kappa\|A^\frac{1}{2}w\|^2\}+\gamma\dual{A^\frac{1}{2}z}{A^\frac{1}{2}w} \\
& -\dual{f_3}{w}-\kappa\dual{A^\frac{1}{2}f_3}{A^\frac{1}{2}w}.
\end{align*}
Taking imaginary part,  using identity \eqref{An005}  and applying Cauchy-Schwarz and Young inequalities and estimates \eqref{dis-10An},   we have
\begin{equation}\label{An006}
\alpha|\lambda|\|Au\|^2\leq C_\tau\|F\|_\mathcal{H}\|U\|_\mathcal{H}.
\end{equation}

 \smallskip
 
 \noindent
{\bf Estimate of $|\lambda|\|z\|^2$.}
Applying the duality product between 
\eqref{esp-40An} and $z$, using the self-adjointness 
of the powers of the operator A,  we have
\begin{align*}
i\lambda\|z\|^2  = & -\alpha\dual{A^\frac{1}{2}v}{A^\frac{1}{2}z}+\gamma\dual{A^\frac{1}{2}w}{A^\frac{1}{2}z} -\delta\|A^\frac{1}{4}z\|^2+\dual{f_4}{z}\\
 =& i\alpha\lambda\|A^\frac{1}{2}v\|^2+\alpha\dual{A^\frac{1}{2}v}{A^\frac{1}{2}f_2}-i\gamma\dual{\sqrt{|\lambda|}A^\frac{1}{2}w}{\dfrac{\lambda}{\sqrt{|\lambda|}}A^\frac{1}{2} v} \\
 &-\gamma\dual{A^\frac{1}{2} w}{A^\frac{1}{2} f_2} -\delta\|A^\frac{1}{4}z\|^2+\dual{f_4}{z}.
\end{align*}
Taking imaginary part,   applying Cauchy-Schwarz and Young inequalities and estimates of $\|U\|_\mathcal{H}^2 $  for $(\theta,\beta)=(\frac{1}{2},1)$ and norms $\|F\|_\mathcal{H}$ and $\|U\|_\mathcal{H}$,  for $\varepsilon>0$ exists $C_\varepsilon>0$ constant that not depend on $\lambda$,  such that
\begin{equation}\label{An007}
|\lambda|\|z\|^2\leq C_\varepsilon\|F\|_\mathcal{H}\|U\|_\mathcal{H}.
\end{equation}
Finally,  from Lemma \eqref{Lemma016} and   estimates \eqref{An005}--\eqref{An007}, we have
\begin{equation}\label{An008}
|\lambda|\|U\|_\mathcal{H}^2\leq C_\varepsilon\|F\|_\mathcal{H}\|U\|_\mathcal{H}\qquad{\rm for}\qquad (\theta,\beta)=\big(\frac{1}{2},1\big).
\end{equation}

\begin{remark}[Asymptotic Behavior]\label{Remark18}
\rm
As mentioned in the introduction, the results of the analyticity class or Gevrey class of a semigroup $S(t)=e^{\mathbb{B}t}$ imply in the exponential decay of this semigroup. We will illustrate this using the following theorem of the spectral characterization of exponential stability of semigroups due to the research of Gearhart\cite{Gearhart}, Pr\"us \cite{Pruss} that the reader may to the    book of Liu and Zheng\cite{LiuZ} (Theorem 1.3.2).  
\begin{theorem}[see \cite{LiuZ}]\label{LiuZExponential}
Let $S(t)=e^{\mathbb{B}t}$ be  a  $C_0$-semigroup of contractions on  a Hilbert space. Then $S(t)$ is exponentially stable if and only if
    \begin{equation}\label{EImaginario}
\rho(\mathbb{B})\supseteq\{ i\lambda/ \lambda\in \R \}  \equiv i\R \qquad {\rm and}
\end{equation}
\begin{equation}\label{Exponential}
 \limsup\limits_{|\lambda|\to
   \infty}   \|(i\lambda I-\mathbb{B})^{-1}\|_{\mathcal{L}(\mathcal{H})}<\infty
\end{equation}
holds.
\end{theorem}

\begin{proof}[\bf Proof.]
We are going to prove to remark \eqref{Remark18} indicated above,  using Theorem \eqref{LiuZExponential},  according to the results obtained in the regularity, note that the propositions \eqref{limsup} and \eqref{EixoIm} are verified in the region $R_E:=\{( \theta,\beta) / 0\leq\theta \leq 1\quad{\rm and}\quad 0<\beta\leq 1\}$ given in {\bf Fig. 08},  then in this region \eqref{EImaginario} and \eqref{Exponential} are verified.  Therefore $S (t)=e^{\mathbb{B}t}$ decays exponentially in the region $R_E$.
 \end{proof}

 \begin{center}
   \begin{tikzpicture}[scale=3, rotate=0]
    \label{Figura08}
 \foreach \x in {0,1}
    \draw (\x cm,1pt) -- (\x cm,-1pt) node[anchor=north] {$\x$};
    \foreach \y in {0,1}
  \draw[thick,->] (0,0) -- (1.25,0)
    node[anchor=north west] {$\theta$ axis};
    \draw[thick,->] (0,0) -- (0,1.25) node[anchor=south east] {$\beta$ axis};
  \draw (-0.1, 1) node[anchor=center] {$1$};
 \coordinate[label=left:] (a) at (0,0);
    \coordinate[label=right:] (b) at (1,0);
    \coordinate[label=right:] (c) at (1,1);
    \coordinate[label=left:] (d) at (0,1);

    \draw[rectangle, fill=black!20!white] (a) -- (b) -- (c) --  (d) -- (a) -- cycle;
    \draw[black!100!black, ultra thick] (0, 0) -- (0, 1);
    \draw[black!100!black, ultra thick] (0, 1) -- (1, 1);
   \draw[white,  dashed] (0, 0) -- (1, 0);
    \draw[black!65!black, ultra thick] (1, 0) -- (1, 1);
    \draw[fill=white!95!black, dashed] (1,0) circle (0.02);
   \draw[fill=white!95!black, dashed] (0,0) circle (0.02);
     \draw[fill=black!95!black, dashed] (1,1) circle (0.02);
       \draw[fill=black!95!black, dashed] (0,1) circle (0.02);
  \end{tikzpicture}

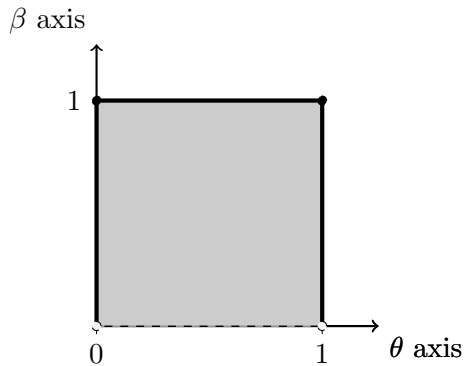
    {\bf Fig.  08}: Exponential Decay Region $R_E$.
\end{center}
     
 \end{remark}

\begin{remark} 
\rm
This research complements the work started by 
\cite{Suarez} that focused on the region 
$R_{SM}:=\{(\theta,\beta)\in [0,1] \times \{0\}\}$,  
they prove that 
$S(t)$ has no analyticity in 
$R_{SM}-\{(1,0)\}$, it is of Gevrey class 
$s>1/\theta$ and they also prove that 
$S (t) $ is analytic em $(\theta,\beta)=(1,0)$.
\end{remark}

\noindent
{\bf Acknowledgments.}
The authors are grateful to the referees 
for their valuable comments and suggestions.

\end{document}